\documentclass[12pt]{amsart}
\usepackage{amssymb}
\usepackage{amsmath}
\usepackage{amssymb}
\usepackage{amsthm}
\usepackage{enumerate}
\usepackage[all,arc,curve,color,frame]{xy}
\usepackage[utf8]{inputenc}

\DeclareMathOperator{\fdeg}{fdeg}

\DeclareMathOperator{\tr}{tr}
\DeclareMathOperator{\ad}{ad}

\newcommand{\A}{{\mathcal A}}
\newcommand{\B}{{\mathcal B}}

\newcommand{\E}{{\mathcal E}}
\newcommand{\F}{{\mathcal F}}

\newcommand{\I}{{\mathcal I}}
\newcommand{\J}{{\mathcal J}}

\newcommand{\K}{{\mathcal K}}
\renewcommand{\L}{{\mathcal L}}

\renewcommand{\O}{{\mathcal O}}
\renewcommand{\P}{{\mathcal P}}
\newcommand{\Q}{{\mathcal Q}}

\newcommand{\C}{\ensuremath{\mathbb{C}}}

\newcommand{\p}{\partial}

\newcommand{\unit}{{\bf{1}}}

\newcommand{\Z}{\ensuremath{\mathbb{Z}}}
\newtheorem{lemma}{Lemma}[section]
\newtheorem{proposition}{Proposition}[section]
\newtheorem{theorem}{Theorem}[section]
\newtheorem{corollary}{Corollary}[section]
\begin{document}

\title[Heat kernel proof of index theorem]
{A heat kernel proof of the index theorem for deformation quantization}
\author[Alexander Karabegov]{Alexander Karabegov}
\address[Alexander Karabegov]{Department of Mathematics, Abilene
Christian University, ACU Box 28012, Abilene, TX 79699-8012}
\email{axk02d@acu.edu}

\subjclass[2010]{53D55, 19K56, 58A50, 35K08}
\keywords{deformation quantization, index theorem, supermanifold, heat kernel}
\date{January 18, 2016}

\maketitle

\begin{center}
{\it This paper is dedicated to my teacher\\
 Alexandre Aleksandrovich Kirillov}
\end{center}

\begin{abstract}
We give a heat kernel proof of the algebraic index theorem for deformation quantization with separation of variables on a pseudo-K\"ahler manifold. We use normalizations of the canonical trace density of a star product and of the characteristic classes involved in the index formula for which this formula contains no extra constant factors.
\end{abstract}

\tableofcontents

\section{Introduction}

Given a manifold $M$, denote by $C^\infty(M)((\nu))$ the space of formal Laurent series
\[
       f = \nu^r f_r + \nu^{r+1} f_{r+1} + \ldots,
\]
where $r \in \Z$ and $f_i \in C^\infty(M)$ for $i \geq r$. We call $f$ a formal function on~$M$.  Let $\pi$ be a Poisson bivector field on $M$. A formal deformation quantization on the Poisson manifold $(M, \pi)$ is an associative product~$\star$ on $C^\infty(M)((\nu))$ given by the formula
\begin{equation}\label{E:star}
              f \star g = fg + \sum_{r=1}^\infty \nu^r C_r(f,g).
\end{equation}
In (\ref{E:star})  $C_r$ are bidifferential operators on $M$ and
\[
             C_1(f,g) - C_1(g,f) = i\{f,g\},
\]
where $\{f,g\} = \pi(df \wedge dg)$ is the Poisson bracket corresponding to $\pi$. We assume that a star product is normalized, i.e., the unit constant function $\unit$  is the identity, $f \star \unit = \unit \star f = f$ for any $f$. Given formal functions $f,g$ on $M$, we denote by $L^\star_f$ the operator of left multiplication by $f$ and by $R^\star_g$ the operator of right multiplication by $g$ with respect to the star product $\star$, so that $f \star g = L^\star_f g = R^\star_g f$. We have $[L^\star_f,R^\star_g]=0$ for any $f,g$.

A star product can be restricted to any open set $U \subset M$. We denote by $C_0^\infty(U)((\nu))$ the space of formal functions compactly supported on~$U$. For  $f = \nu^r f_r + \ldots \in C_0^\infty(U)((\nu))$ each function $f_i$ has compact support in $U$, but we do not requre that all $f_i, i \geq r$, have a common compact support in $U$.

Two star products $\star_1$ and $\star_2$ on $(M, \pi)$ are equivalent if there exists a formal differential operator
\[
                   T = 1 + \nu T_1 + \nu^2 T_2 + \ldots
\]
on $M$ such that
\[
     f \star_2 g = T^{-1} (Tf \star_1 Tg).
\]
The problem of existence and classification up to equivalence of star products on Poisson manifolds was stated in \cite{BFFLS} and settled by Kontsevich in \cite{K}, who proved that star products exist on an arbitrary Poisson manifold and their equivalence classes are parametrized by the formal deformations of the Poisson structure.

A  symplectic manifold $(M, \omega_{-1})$ is equipped with a nondegenerate Poisson bivector field $\pi$ inverse to $\omega_{-1}$. Fedosov gave in \cite{F1}  and \cite{F2} a simple geometric construction of star products in each equivalence class on an arbitrary symplectic manifold $(M, \omega_{-1})$. The equivalence classes of star products on $(M, \omega_{-1})$ are bijectively parametrized by the formal cohomology classes from
\[
         -\frac{i}{\nu}\left[\omega_{-1}\right] + H^2(M)[[\nu]],
\]
as shown in \cite{D}, \cite{NT2}, \cite{F2}, \cite{BCG}.

Let $\star$ be a star product on a connected symplectic manifold $(M, \omega_{-1})$ of dimension $2m$. There exists a formal trace density for the product~$\star$ which is globally defined on $M$ and is unique up to a factor from $\C((\nu))$. Fedosov introduced in \cite{F1} a canonically normalized formal trace density for the product $\star$ using local isomorphisms between that product and the Moyal-Weyl star product. Then in \cite{F2} he used this trace density to state and prove the algebraic index theorem for the star product $\star$.

In this paper we consider a canonical formal trace density $\mu_\star$ of the product $\star$ on $M$ which differs from Fedosov's trace density by a factor from $\C$. According to \cite{LMP2}, the normalization of $\mu_\star$ can be described intrinsically as follows.  On each contractible open subset $U \subset M$ there exists a local $\nu$-derivation of the product $\star$ of the form
\[
                    \delta_\star = \frac{d}{d\nu} + A,            
\]
where $A$ is a formal differential operator on $U$ (see \cite{GR99}). It is unique up to an inner derivation, i.e., all such $\nu$-derivations on $U$ are of the form $\delta_\star + [f, \cdot]_\star$, where $f \in C^\infty(U)((\nu))$ and $[\cdot,\cdot]_\star$ is the commutator with respect to the product~$\star$. The canonical trace density $\mu_\star$ satisfies the equation
\begin{equation}\label{E:nuderiv}
       \frac{d}{d\nu} \int_U f \mu_\star = \int_U \delta_\star(f) \mu_\star
\end{equation}
for any formal function $f$ compactly supported on $U$.  Equation ~(\ref{E:nuderiv}) determines $\mu_\star$ on $U$ up to a factor from $\C$ which can be fixed by normalizing the leading term of $\mu_\star$. In this paper we require that this leading term be
\begin{equation}\label{E:lead}
     \frac{1}{m!} \left(-\frac{i}{\nu} \omega_{-1}\right)^m.
\end{equation}

If $M$ is compact,  the total volume of the canonical trace density $\mu_\star$ is given by a topological formula analogous to the Atiyah-Singer formula for the index of an elliptic operator,
\begin{equation}\label{E:index}
                   \int_M \mu_\star =  \int_M e^{\theta_\star} \hat {A}(M),
\end{equation}
where $\theta_\star$ is the formal cohomology class that parametrizes the equivalence class of the star product $\star$ and $\hat {A}(M)$ is the $\hat A$-genus of the manifold $M$. The class $\hat {A}(M)$ has a de Rham representative
\[
     {\det} ^{\frac{1}{2}} \frac{R_{TM}/2}{\sinh (R_{TM}/2)},
\]
where $R_{TM}$ is the curvature of an arbitrary connection on $TM$. This statement is called the algebraic index theorem for deformation quantization and the total volume of the canonical trace density $\mu_\star$ is called the algebraic index of the star product $\star$. The algebraic index theorem has several different conceptual proofs. Fedosov's proof is based upon the methods of Atiyah and Singer. Nest and Tsygan proved in \cite{NT1}  the algebraic index theorem for deformation quantization using cyclic homology and the local Riemann-Roch theorem by Feigin and Tsygan given in \cite{FT}. Various generalizations of the algebraic index theorem were obtained in \cite{NT2}, \cite{FFS}, \cite{DR}, \cite{PPT1}, \cite{PPT2}.

Getzler gave in \cite{G} a proof of the Atiah-Singer index theorem for a Dirac operator based upon the ideas of Witten and Alvarez-Gaum\' e (see \cite{AG}). In that proof he used symbols of pseudodifferential operators on a supermanifold. Berline, Getzler, and Vergne wrote later a book~\cite{BGV} on heat kernel proofs of index theorems for Dirac operators.

In this paper we prove the algebraic index theorem for a star product with separation of variables on a pseudo-K\"ahler manifold following Getzler's approach. Many global geometric objects used in our proof are described locally on holomorphic coordinate charts by coordinate-independent constructions. The proofs of a number of statements are based on the interplay between pointwise products and star products with separation of variables. We use normalizations of the canonical trace density and of the characteristic classes involved in the index formula~(\ref{E:index}) for which this formula contains no extra constant factors.

\medskip

This paper is dedicated to my teacher Alexandre Aleksandrovich Kirillov on the occasion of his 81st birthday.

\section{Star products with separation of variables}
Let $M$ be a complex manifold of complex dimension $m$ equipped with a Poisson bivector field $\pi$. A star product $\star$ on $(M, \pi)$ has the property of separation of variables of the anti-Wick type if
\[
      a \star f = af \mbox{ and } f \star b = fb
\]
for any locally defined holomorphic function $a$, antiholomorphic function $b$, and arbitrary function $f$, which means that
\[
                    L^\star_a = a \mbox{ and } R^\star_b = b
\]
are pointwise multiplication operators. Equivalently, the operators $C_r$ in (\ref{E:star}) act on the first argument by antiholomorphic partial derivatives and on the second argument by holomorphic ones. If there exists a star product with separation of variables on $(M,\pi)$, then the Poisson bivector~$\pi$ is of type (1,1) with respect to the complex structure. In local coordinates $\pi$ is expressed as follows,
\[
      \pi = i g^{\bar lk} \frac{\p}{\p z^k} \wedge \frac{\p}{\p \bar z^l},
\]
where $g^{\bar lk}$ is the Poisson tensor corresponding to $\pi$ and the Einstein summation over repeated upper and lower indices is used.

We say that a formal differential operator $A = A_0 + \nu A_1 + \ldots$ on a manifold $M$ is natural if $A_r$ is a differential operator on $M$ of order not greater than $r$ for all $r \geq 0$. A star product (\ref{E:star}) is natural in the sense of~\cite{GR} if the bidifferential operator $C_r$ in (\ref{E:star}) is of order not greater than ~$r$ in each argument for all $r \geq 1$ or,  equivalently, if for any $f \in C^\infty(M)$ the operators $L^\star_f$ and $R^\star_f$ on $M$ are natural. If a star product $\star$ is natural and $f = \nu^p f_p + \nu^{p+1} f_{p+1} +\ldots \in C^\infty(M)((\nu))$ (i.e., the $\nu$-filtration degree of $f$ is at least $p$), then the operators $\nu^{-p} L^\star_f$ and $\nu^{-p} R^\star_f$ are natural. It was proved in \cite{CMP3} that any star product with separation of variables on a complex manifold~$M$ is natural.

Given a star product with separation of variables $\star$ on $(M,\pi)$, there exists a unique globally defined formal differential operator
\[
               \I_\star = 1 + \nu \I_1 + \nu^2 \I_2 + \ldots
\]
on $M$ such that for any locally defined holomorphic function $a$ and antiholomorphic function $b$,
\[
         \I_\star(ba) = b \star a.
\]
In particular, $\I_\star a = a$ and $\I_\star b = b$. It is called the formal Berezin transform associated with the star product $\star$. A star product with separation of variables can be recovered from its formal Berezin transform. An equivalent star product $\star'$ on $(M, \pi)$ given by the formula
\[
          f \star' g = \I_\star^{-1} (\I_\star f \star \I_\star g)
\]
is a star product with separation of variables of the Wick type, so that
\[
                b \star' f = bf \mbox{ and } f \star' a = fa,
\]
where $a$ and $b$ are as above.
\begin{lemma}
For any local holomorphic function $a$ and local antiholomorphic function $b$ we have
\begin{equation}\label{E:bertran}
  \I_\star(f a)  = \I_\star(f ) \star a \mbox{ and } \I_\star(bf) = b \star \I_\star(f ).
\end{equation}
\end{lemma}
\begin{proof}
\[
     \I_\star(f a) = \I_\star(f \star' a) = \I_\star(f) \star \I_\star(a) = \I_\star(f ) \star a.
\]
The second formula can be proved similarly.
\end{proof}
The star product $\tilde\star$ opposite to $\star'$,
\[
     f\,  \tilde \star \,  g = g \star' f,
\]
is a star product with separation of variables of the anti-Wick type on~$(M, -\pi)$. The star product~$\tilde\star$ is called dual to $\star$. Its formal Berezin transform is
$\I_{\tilde\star} = \I_\star^{-1}.$

In this paper we assume that a star product with separation of variables is of the anti-Wick type unless otherwise specified.

Let $\star$ be a star product with separation of variables on $(M, \pi)$. The operator $C_1$ in (\ref{E:star}) written in coordinates on a local chart $U \subset M$ is of the form
\[
             C_1(f,g) = g^{\bar lk} \frac{\p f}{\p \bar z^l}\frac{\p g}{\p z^k},
\]
where $g^{\bar lk}$ is the Poisson tensor corresponding to $\pi$. If $\pi$ is nondegenerate, it corresponds to a pseudo-K\"ahler form $\omega_{-1}$ on $M$. Namely, the matrix $g_{k\bar l}$ inverse to $g^{\bar lk}$  is a pseudo-K\"ahler metric tensor such that
\[
    \omega_{-1} =     i g_{k\bar l} dz^k \wedge d\bar z^l
\]
on $U$. If $\Phi_{-1}$ is a potential of $\omega_{-1}$ on $U$, then
\begin{equation}\label{E:metric}
                g_{k\bar l} = \frac{\p^2 \Phi_{-1}}{\p z^k \p \bar z^l}.
\end{equation}
We will omit the bars over the antiholomorphic indices  in the tensors $g_{k\bar l}$ and $g^{\bar lk}$. In this paper we will use the notation
\[
      g_{k_1 \ldots k_p \bar l_1 \ldots \bar l_q} = \frac{\p^{p+q} \Phi_{-1}}{\p z^{k_1} \ldots \p z^{k_p} \p \bar z^{l_1} \ldots \bar \p z^{l_q}}
\]
for $p, q \geq 1$. 

It was proved in \cite{CMP1} and \cite{BW} that star products with separation of variables exist on an arbitrary pseudo-K\"ahler manifold $(M,\omega_{-1})$. Moreover, as shown in \cite{CMP1}, the star products with separation of variables of the anti-Wick type on $(M,\omega_{-1})$ bijectively correspond to the closed formal (1,1)-forms
\begin{equation}\label{E:class}
      \omega = \nu^{-1}\omega_{-1} + \omega_0 + \nu \omega_1 + \ldots
\end{equation}
on $M$. Let $\star$ be a star product with separation of variables on $(M, \omega_{-1})$ with classifying form (\ref{E:class}). On a contractible coordinate chart $U \subset M$ every closed form $\omega_r$ has a potential $\Phi_r$, so that $\omega_r = i \p \bar \p \Phi_r$. Then
\[
     \Phi : = \nu^{-1}\Phi_{-1} + \Phi_0 + \nu \Phi_1 + \ldots
\]
is a formal potential of $\omega$ on $U$. The star product  $\star$ is uniquely determined by the property that
\begin{equation}\label{E:classleft}
           L^\star_{\frac{\p\Phi}{\p z^k}} = \frac{\p\Phi}{\p z^k} + \frac{\p}{\p z^k} \mbox{ and } R^\star_{\frac{\p\Phi}{\p \bar z^l}} = \frac{\p\Phi}{\p \bar z^l} + \frac{\p}{\p \bar z^l}
\end{equation}
for $1 \leq k,l \leq m$.
 Given $f \in C^\infty(U)((\nu))$, there exists a unique formal differential operator $A$ on $U$ which commutes with the operators $R^\star_{\bar z^l} = \bar z^l$ and $R^\star_{\p\Phi / \p \bar z^l}$ for $1 \leq l \leq m$ and satisfies the condition $A\unit = f$. It coincides with the operator~$L^\star_f, \ A = L^\star _f$. This property allows to reconstruct the star product~ $\star$ from its classifying form $\omega$.

The Ricci form $\rho$ on $(M,\omega_{-1})$ is given in local coordinates by the formula
\[
         \rho = - i \p \bar \p \det (g_{kl}).
\]
The canonical class $\varepsilon_M$ of the complex manifold $M$ has a de Rham representative $- \rho, \ \varepsilon_M = - [\rho]$.

The formal cohomology class $\theta_\star$ that parametrizes the equivalence class of a star product with separation of variables $\star$ on $M$ with classifying form $\omega$ is given by the formula
\begin{equation}\label{E:class1}
    \theta_\star = -i \left([\omega] - \frac{1}{2}\varepsilon_M\right) = -i \left([\omega] + \frac{1}{2}[\rho]\right),
\end{equation}
where $[\omega]$ is the de Rham class of $\omega$. Formula (\ref{E:class1}) was given in \cite{LMP1}, but, unfortunately, contained a wrong sign.

Let $\tilde\omega$ be the classifying form of the star product with separation of variables $\tilde\star$ on $(M, -\omega_{-1})$ dual to the product $\star$. Then 
\[
       \tilde \omega = -\nu^{-1}\omega_{-1} + \tilde\omega_0 + \nu \tilde\omega_1 + \ldots.
\]

The following construction of a local non-normalized trace density for a star product with separation of variables $\star$ on $(M,\omega_{-1})$ was introduced in~ \cite{LMP2}. Given a contractible coordinate chart $U \subset M$ and a potential $\Phi = \nu^{-1} \Phi_{-1} + \Phi_0 + \ldots$ of the classifying form $\omega$ of the product~$\star$ on~$U$, there exists a potential~$\Psi = -\nu^{-1} \Phi_{-1} + \Psi_0 + \nu \Psi_1 + \ldots$ of the dual form $\tilde\omega$ on~$U$ satisfying the equations

\begin{equation}\label{E:altdens}
      \I_\star \left(\frac{\p \Psi}{\p z^k}\right) + \frac{\p \Phi}{\p z^k} = 0 \mbox{ and }  \I_\star \left(\frac{\p \Psi}{\p \bar z^l}\right) + \frac{\p \Phi}{\p \bar z^l} = 0.
\end{equation}
The potential $\Psi$ is determined by equations (\ref{E:altdens}) up to an additive formal constant.  As shown in \cite{LMP2},
\[
                      e^{\Phi + \Psi} dz d\bar z,
\] 
where $dz d\bar z$ is a Lebesgue measure on $U$, is a trace density for the product $\star$ on $U$. In order to canonically normalize this trace density, one can use the following explicit local $\nu$-derivation of the product $\star$ on $U$ introduced in \cite{LMP1},
\begin{equation}\label{E:separdelta}
             \delta_\star = \frac{d}{d\nu} + \frac{d\Phi}{d \nu} - R^\star_{\frac{d\Phi}{d \nu}}.
\end{equation}

\section{Deformation quantization on a super-K\"ahler manifold}

In this section we recall a construction of a star product with separation of variables on a split supermanifold from \cite{JGPsubm}. 

Let $E$ be a holomorphic vector bundle of rank $d$ over a pseudo-K\"ahler manifold $(M, \omega_{-1})$ equipped with a possibly indefinite sesquilinear fiber metric $h_{\alpha\bar\beta}$ and let $\Pi E$ be the corresponding split supermanifold. We identify the functions on $\Pi E$ with the sections of $\wedge \left(E^\ast \oplus \bar E^\ast \right)$, where $E^\ast$ and $\bar E$ are the dual and the conjugate bundles of $E$, respectively.

We say that a formal function $f = \nu^r f_r + \ldots$ on $\Pi E$ is compactly supported {\it over} $M$ if for each $j \geq r$ the coefficient $f_j$ is a compactly supported section of $\wedge \left(E^\ast \oplus \bar E^\ast \right)$, or, equivalently, there exists a compact $K_j \subset M$ such that the restriction of the function $f_j$ to $\Pi E|_{M \setminus K_j}$ vanishes.

Consider a holomorphic trivialization $E|_U \cong U \times \C^d$ over an open set $U \subset M$ and denote by $\theta^\alpha, \bar \theta^\beta, 1 \leq \alpha,\beta \leq d,$ the odd fiber coordinates on $\Pi E|_U \cong U \times \C^{0|d}$. A function $f$ on $\Pi E|_U$ can be written as
\begin{equation}\label{E:funcf}
    f =  \sum_{0 \leq p, q \leq d} f_{\alpha_1 \ldots \alpha_p \bar\beta_1 \ldots \bar\beta_q} \theta^{\alpha_1} \ldots \theta^{\alpha_p} \bar \theta^{\beta_1} \ldots \bar \theta^{\beta_q},
\end{equation}
where the coefficients $f_{\alpha_1 \ldots \alpha_p \bar\beta_1 \ldots \bar\beta_q} \in C^\infty(U)$ are separately antisymmetric in the indices $\alpha_i$ and $\beta_j$. A function (\ref{E:funcf}) on $\Pi E|_U$ is holomorphic if its coefficients  are holomorphic and satisfy $f_{\alpha_1 \ldots \alpha_p \bar\beta_1 \ldots \bar\beta_q}=0$ for $q >0$. It is antiholomorphic if its coefficients are antiholomorphic and satisfy  $f_{\alpha_1 \ldots \alpha_p \bar\beta_1 \ldots \bar\beta_q}=0$ for $p >0$.

The fiber metric $h_{\alpha\bar\beta}$ on $E$ determines a global even nilpotent function $H = \nu^{-1} H_{-1}$ on $\Pi E$ such that locally
\[
                        H_{-1} = h_{\alpha\bar\beta} \theta^\alpha \bar \theta^\beta.
\]
Let $\star$ be a star product with separation of variables on $(M, \omega_{-1})$ with classifying form $\omega$. It was shown in \cite{JGPsubm} that the star product $\star$ and the function $H$ determine a unique global star product with separation of variables $\ast$ on $\Pi E$ which is $\Z_2$-graded with respect to the standard parity of the functions on $\Pi E$ and satisfies the following property. Let $U \subset M$ be any contractible coordinate chart, $\Pi E|_U \cong U \times \C^{0|d}$ be a trivialization, $\Phi = \nu^{-1} \Phi_{-1} + \Phi_0 + \ldots$ be a potential of the form $\omega$ on~$U$ identified with its lift to $\Pi E|_U$, and 
\begin{equation}\label{E:hmin}
     X:= \Phi + H = \nu^{-1}(\Phi_{-1} +H_{-1}) + \Phi_0 + \nu \Phi_1 + \ldots
\end{equation}
be an even superpotential on $\Pi E|_U$. Then
\begin{eqnarray*}
        L_{\frac{\p X}{\p z^k}} = \frac{\p X}{\p z^k} + \frac{\p}{\p z^k},  L_{\frac{\p X}{\p \theta^\alpha}} = \frac{\p X}{\p \theta^\alpha} + \frac{\p}{\p \theta^\alpha},\hskip 2cm \\
 R_{\frac{\p X}{\p \bar z^l}} = \frac{\p X}{\p  \bar z^l} + \frac{\p}{\p  \bar z^l},  \mbox{ and } R_{\frac{\p X}{\p \bar \theta^\beta}} = \frac{\p X}{\p \bar \theta^\beta} + \frac{\p}{\p \bar \theta^\beta}.\nonumber
\end{eqnarray*}
Here we assume that the fiberwise Grassmann multiplication operators and partial derivatives with respect to the odd variables $\theta, \bar\theta$ act from the left, $L_f$ is the left $\ast$-multiplication operator by $f$ so that $L_f g = f \ast g$, and $R_f$ is the graded right $\ast$-multiplication operator by $f$, so that if $f$ and $g$ are homogeneous functions  on $\Pi E$, then
\begin{equation}\label{E:grr}
             R_f g = (-1)^{|f||g|} g \ast f.    
\end{equation}
In particular, $L_f$ supercommutes with $R_g$ for any $f,g$. The star product ~$\ast$ on $\Pi E|_U$ is determined by the potential $X$. Given a formal function $f \in C^\infty(\Pi E|_U)((\nu))$, one can describe the operator $L_f$ as follows. There exists a unique  formal differential operator $A$ on $\Pi E|_U$ which supercommutes with the operators 
\[
     R_{\bar z^l} = \bar z^l, R_{\bar\theta^\beta} = \bar\theta^\beta, R_{\frac{\p X}{\p\bar z^l}} = \frac{\p X}{\p\bar z^l} + \frac{\p}{\p\bar z^l}, \mbox{ and } R_{\frac{\p X}{\p\bar \theta^\beta}} = \frac{\p X}{\p\bar \theta^\beta} + \frac{\p}{\p\bar \theta^\beta}
\]
and is such that $A\unit = f$. It coincides with the operator $L_f, A = L_f$.

Denote by $\I$ the formal Berezin transform for the product $\ast$. It is a formal differential operator globally defined on $\Pi E$ and such that
\[
              \I(ba) = b \ast a
\] 
for any local holomorphic function $a$ and antiholomorphic function $b$ on $\Pi E$. In particular, $\I a = a$ and $\I b = b$. One can prove formulas analogous to (\ref{E:bertran}) for the operator $\I$. For any function $f$,
\begin{equation}\label{E:sbertran}
         \I(f a)  = \I(f ) \ast a \mbox{ and } \I(bf) = b \ast \I(f ).
\end{equation}

It was shown in \cite{JGPsubm} that the star product $\ast$ has a supertrace given by a canonically normalized formal supertrace density globally defined on $\Pi E$. 

A local non-normalized supertrace density for the product $\ast$ can be obtained as follows. Given a contractible coordinate chart $U \subset M$ and a superpotential (\ref{E:hmin}) which determines the star product $\ast$ on $\Pi E|_U$, there exists an even superpotential
\[
\tilde X = -\nu^{-1}(\Phi_{-1} + H_{-1}) + \tilde X_0 + \nu \tilde X_1 + \ldots 
\]
on $\Pi E|_U$ satisfying the equations
\begin{eqnarray*}
      \I \left(\frac{\p \tilde X}{\p z^k}\right) + \frac{\p X}{\p z^k} = 0,\  \I \left(\frac{\p\tilde X}{\p \theta^\alpha}\right) + \frac{\p X}{\p \theta^\alpha} = 0,\\
  \I \left(\frac{\p \tilde X}{\p \bar z^l}\right) + \frac{\p X}{\p \bar z^l} = 0, \mbox{ and }  \I \left(\frac{\p \tilde X}{\p \bar \theta^\beta}\right) + \frac{\p X}{\p \bar \theta^\beta} = 0.
\end{eqnarray*}
The formula
\begin{equation}\label{E:strdens}
                    \mathrm{e}^{X + \tilde X} dz d\bar z d\theta d\bar\theta,
\end{equation}
where $dz d\bar z$ is a Lebesgue measure on $U$ and $d\theta d\bar\theta$ is a Berezin density (coordinate volume form) on $\C^{0|d}$, gives a supertrace density for the star product $\ast$ on $\Pi E|_U \cong U \times \C^{0|d}$. It is determined up to a multiplicative formal constant.

\section{A star product on $TM \oplus \Pi TM$}\label{S:tmpitm}

In this section we fix a star product with separation of variables $\star$ with classifying form $\omega$ on a pseudo-K\"ahler manifold $(M, \omega_{-1})$ of complex dimension $m$. We recall a construction from \cite{CMP4} of a star product with separation of variables $\bullet$ on the tangent bundle $TM$ obtained from the product $\star$. We use the product $\bullet$ to construct a star product $\ast$ on the supermanifold $TM \oplus \Pi TM$ which will be the main framework for the proof of the algebraic index theorem for the star product $\star$.

The tangent bundle $TM$ can be identified with the cotangent bundle $T^\ast M$ via the pseudo-K\"ahler metric on $M$. It was shown in \cite{CMP4} that the canonical symplectic form on $T^\ast M$ transferred to $TM$ via this identification is a global pseudo-K\"ahler form $\Xi_{-1}$ on $TM$. Let $U \subset M$ be a contractible coordinate chart with coordinates $z^k, \bar z^l, 1 \leq k,l \leq m,$ and $\Phi_{-1}$ be a potential of $\omega_{-1}$ on $U$. Denote by $\eta^k, \bar\eta^l$ the corresponding fiber coordinates on $TU$. Then
\[
               \frac{\p \Phi_{-1}}{\p z^k} \eta^k +  \frac{\p \Phi_{-1}}{\p \bar z^l} \bar\eta^l
\]
is a potential of $\Xi_{-1}$ on $TU$. Let $\pi_{TM}:TM \to M$ be the natural projection. It was shown in \cite{CMP4} that
\[
              \Omega_{-1} :=    \pi_{TM}^\ast \omega_{-1} + \Xi_{-1}
\]
is also a global pseudo-K\"ahler form on $TM$. We denote by $\bullet$  the star product with separation of variables on the pseudo-K\"ahler manifold $(TM,\Omega_{-1})$ with the classifying form
\[
             \Omega :=    \pi_{TM}^\ast \omega + \nu^{-1}\Xi_{-1}.
\]
If $\Phi$ is a potential of $\omega$ on $U$, then
\[
             \Phi + \nu^{-1}\left(\frac{\p \Phi_{-1}}{\p z^k} \eta^k +  \frac{\p \Phi_{-1}}{\p \bar z^l} \bar\eta^l\right)
\]
is a potential of the form $\Omega$ on $TU$.

Denote by $E$ the holomorphic vector bundle over $TM$ which is the pullback of the holomorphic tangent bundle $T^{(1,0)}M$ by the natural projection $\pi_{TM}$,
\[
                         E := \pi^\ast_{TM} \left(T^{(1,0)}M\right).
\]
We  equip $E$ with the fiber metric induced by the pseudo-K\"ahler metric $g_{kl}$ on $M$. The split supermanifold $\Pi E$ can be identified with the total space of the bundle $TM \oplus \Pi TM \to M$. Let $U \subset M$ be a coordinate chart and
\[
              TU \oplus \Pi TU \cong U \times \C^{m|m}
\]
be the corresponding trivialization. We denote as above  the even fiber coordinates by $\eta^k, \bar \eta^l$ and the odd ones by $\theta^k, \bar\theta^l$.  

Let $\psi = \nu^{-1} \psi_{-1}$ be a global even nilpotent function on $TM \oplus \Pi TM$ such that in local coordinates
\begin{equation}\label{E:psimin}
        \psi_{-1} = \theta^k g_{kl} \bar\theta^l. 
\end{equation}
As shown in \cite{JGPsubm}, there is a unique star product with separation of variables $\ast$ on $TM \oplus \Pi TM \cong \Pi E$ obtained from the product ~$\bullet$ on $TM$ and the function~ $\psi$ such that for any contractible coordinate chart $U \subset M$ the product $\ast$ is determined on $TU \oplus \Pi TU$ by the even superpotential
\begin{equation}\label{E:superpot}
        X = \Phi + \nu^{-1}\left(\frac{\p \Phi_{-1}}{\p z^p}\eta^p + \frac{\p \Phi_{-1}}{\p \bar z^l}\bar \eta^l + g_{kl} \theta^k \bar \theta^l \right).
\end{equation}
We denote by $L_f$ and $R_f$ the left and the graded right $\ast$-multiplication operators by a function $f$, respectively.  The following formulas hold on $TU \oplus \Pi TU$,
\begin{eqnarray}\label{E:superlr}
        L_{\frac{\p X}{\p z^k}} = \frac{\p X}{\p z^k} + \frac{\p}{\p z^k}, R_{\frac{\p X}{\p \bar z^l}} = \frac{\p X}{\p  \bar z^l} + \frac{\p}{\p  \bar z^l},\nonumber \\
L_{\frac{\p X}{\p \eta^k}} = \frac{\p X}{\p \eta^k} + \frac{\p}{\p \eta^k}, R_{\frac{\p X}{\p \bar \eta^l}} = \frac{\p X}{\p  \bar \eta^l} + \frac{\p}{\p  \bar \eta^l},\\
L_{\frac{\p X}{\p \theta^k}} = \frac{\p X}{\p \theta^k} + \frac{\p}{\p \theta^k},
  \mbox{ and } R_{\frac{\p X}{\p \bar \theta^l}} = \frac{\p X}{\p \bar \theta^l} + \frac{\p}{\p \bar \theta^l}.\nonumber
\end{eqnarray}

We introduce two families of operators on $TU \oplus \Pi TU$,
\[
      D^k = g^{qk} \frac{\p}{\p \bar z^q} \mbox{ and } \bar D^l = g^{lp} \frac{\p}{\p z^p}, \ 1 \leq k,l \leq m.
\]
It is known that $[D^k, D^p]=0$ and $[\bar D^l, \bar D^q] = 0$ for all $k,l,p,q$.
\begin{proposition}\label{P:superpot}
Given a formal function $f \in C^\infty(U)((\nu))$ identified with its lift to $TU \oplus \Pi TU$, the following formulas hold,
\begin{equation}\label{E:lbullf}
       L_f = \sum_{r = 0}^\infty \frac{\nu^r}{r!} \left(D^{k_1} \ldots D^{k_r} f\right) \frac{\p^r}{\p \eta^{k_1} \ldots \p \eta^{k_r}}
\end{equation}
and
\begin{equation}\label{E:rbullf}
R_f = \sum_{r = 0}^\infty \frac{\nu^r}{r!} \left(\bar D^{l_1} \ldots \bar D^{l_r} f\right) \frac{\p^r}{\p \bar\eta^{l_1} \ldots \p \bar\eta^{l_r}}.
\end{equation}
\end{proposition} 
\begin{proof}
Denote temporarily the operator on the right-hand side of (\ref{E:lbullf}) by ~$A$. Since $A$ is even, for any operator $B$ the commutator of~$A$ and~$B$ coincides with their supercommutator. Clearly, $A1=f$ and ~$A$ commutes with the fiberwise Grassmann multiplication operators by the functions on $TU \oplus \Pi TU$ which do not depend on the variables~$\eta$ and with the operators
\begin{eqnarray*}
      R_{\frac{\p X}{\p \bar \eta^l}}= \frac{\p X}{\p \bar \eta^l} + \frac{\p}{\p\bar \eta^l} = \nu^{-1}\frac{\p \Phi_{-1}}{\p \bar z^l} + \frac{\p}{\p\bar \eta^l} \mbox{ and }\\
R_{\frac{\p X}{\p \bar \theta^l}}= \frac{\p X}{\p \bar \theta^l} + \frac{\p}{\p\bar \theta^l} = -\nu^{-1}g_{kl}\theta^k + \frac{\p}{\p\bar \theta^l}.
\end{eqnarray*}
In order to prove formula (\ref{E:lbullf}) it remains to show that $A$ commutes with the operators
\begin{eqnarray*}
 R_{\frac{\p X}{\p \bar z^l}} = \frac{\p X}{\p \bar z^l} + \frac{\p}{\p \bar z^l} = \frac{\p\Phi}{\p \bar z^l} + \nu^{-1} \frac{\p^2 \Phi_{-1}}{\p \bar z^l \p \bar z^q}\bar \eta^q +\\
 \nu^{-1} g_{pl}\left(\eta^p + \nu D^p  \right) + \nu^{-1} g_{kp \bar q}\theta^p \bar \theta^q.
\end{eqnarray*}
Since $A$ commutes with the multiplication operators by the functions on $TU \oplus \Pi TU$ which do not depend on the variables $\eta$, it sufficies to prove that it commutes with the operators $\eta^p + \nu D^p$. We will consider a ``Fourier transform" which maps the operator $\p / \p \eta^k$ to the multiplication operator by the variable $\xi_k$ and the multiplication operator by $\eta^k$ to the operator $-\p/\p \xi^k$. This mapping extends to an isomorphism from the algebra of polynomial differential operators in the variables $\eta^k$ onto that in the variables $\xi_k$. The operator $A$ will be mapped to the multiplication operator by the function $\{\exp(\nu \xi_k D^k)\}f$ and the operator $\eta^p + \nu D^p$ will be mapped to
\begin{equation}\label{E:fourier}
         - \frac{\p}{\p \xi_p} + \nu D^p.
\end{equation}
It is clear that the operators $\{\exp(\nu \xi_k D^k)\}f$ and (\ref{E:fourier}) commute, which concludes the proof of formula (\ref{E:lbullf}). Formula (\ref{E:rbullf}) can be proved similarly.
\end{proof}
Denote by $\I$ the formal Berezin transform of the star product~ $\ast$.
\begin{corollary}\label{C:fastg}
Given functions $f,g$ on $M$ identified with their lifts to $TM \oplus \Pi TM$ via the natural projection, we have
\begin{equation}\label{E:fbullg}
      f \ast g = fg.
\end{equation}
Also, $\I f = f$.
\end{corollary}
\begin{proof}
Formula (\ref{E:fbullg}) follows from (\ref{E:lbullf}). Let $U \subset M$ be a coordinate chart and $a = a(z)$ and $b = b(\bar z)$ be a holomorphic and an antiholomorphic function on $U$, respectively, lifted to $TU \oplus \Pi TU$. Then formula (\ref{E:fbullg}) implies that
\[
              \I (ba) = b \ast a = ba.
\]
Since $\I$ is a formal differential operator, it follows that $\I  f = f$ for any function $f$ on $M$ lifted to $TM \oplus \Pi TM$.
\end{proof}
Let $U \subset M$ be a contractible coordinate chart. We set $\mathbf{g} := \det (g_{kl})$ and denote by $\log \mathbf{g}$ any branch of the logarithm of $\mathbf{g}$ on $U$.  Below we will calculate a supertrace density for the star product $\ast$ on $TU \oplus \Pi TU$. 
\begin{lemma}\label{L:phitheta}
The following formula holds,
\[
     \bar \theta^q \ast \left(\nu^{-1}g_{kp \bar q}\right) \ast \theta^p  = - \nu^{-1}g_{kp \bar q}\theta^p \bar \theta^q  + g^{qp} g_{kp \bar q}.
\]
\end{lemma} 
\begin{proof}
Using Proposition \ref{P:superpot} and a formula from (\ref{E:superlr}),
\begin{equation}\label{E:lbartheta}
     L_{\frac{\p X}{\p \theta^s}} =L_{\nu^{-1} g_{sq}\bar \theta^q} = \nu^{-1} g_{sq}\bar \theta^q + \frac{\p}{\p \theta^s},
\end{equation}
we obtain the statement of the lemma from the calculation
\begin{eqnarray*}
    \bar \theta^q \ast \left(\nu^{-1}g_{kp \bar q}\right) \ast \theta^p  = \bar \theta^q \ast (\nu^{-1} g_{sq}) \ast g^{ts} \ast g_{kp \bar t} \ast \theta^p =\\
 (\nu^{-1} g_{sq}\bar \theta^q) \ast \left(g^{ts}g_{kp \bar t}\right) \ast \theta^p  =
(\nu^{-1} g_{sq}\bar \theta^q) \ast \left(g^{ts}g_{kp \bar t}\theta^p\right)   =\\  (\nu^{-1} g_{sq}\bar \theta^q) \left(g^{ts}g_{kp \bar t}\theta^p\right) + 
g^{ts}g_{ks \bar t}.
\end{eqnarray*}
\end{proof}
\begin{proposition}\label{E:superdens}
The even superpotential $X' : = - X + \log \mathbf{g}$ satisfies the following equations.
\begin{eqnarray}\label{E:superlong}
\frac{\p X}{\p z^k} + \I\left(\frac{\p X'}{\p z^k}\right) =0, \frac{\p X}{\p \bar z^l} + \I\left(\frac{\p X'}{\p \bar z^l}\right) =0,\nonumber \\
\frac{\p X}{\p \eta^k} + \I\left(\frac{\p X'}{\p \eta^k}\right) =0, \frac{\p X}{\p \bar\eta^l} + \I\left(\frac{\p X'}{\p \bar\eta^l}\right) =0,\\
\frac{\p X}{\p \theta^k} + \I\left(\frac{\p X'}{\p \theta^k}\right) =0, \mbox{ and } \frac{\p X}{\p \bar\theta^l} + \I\left(\frac{\p X'}{\p \bar\theta^l}\right) =0.\nonumber
\end{eqnarray}
\end{proposition}
\begin{proof}
Using formulas (\ref{E:sbertran}), Proposition \ref{P:superpot}, and Lemma \ref{L:phitheta}, we get that
\begin{eqnarray*}
\frac{\p X}{\p z^k} + \I\left(\frac{\p X'}{\p z^k}\right) = \nu^{-1}\frac{\p^2 \Phi_{-1}}{\p z^k \p z^p}\eta^p + \nu^{-1} g_{kq}\bar \eta^q + \nu^{-1}g_{kp \bar q}\theta^p \bar \theta^q +\\
\I\left(-\nu^{-1}\frac{\p^2 \Phi_{-1}}{\p z^k \p z^p}\eta^p - \nu^{-1} g_{kq}\bar \eta^q - \nu^{-1}g_{kp \bar q}\theta^p \bar \theta^q + \frac{\p \log \mathbf{g}}{\p z^k}\right) =\\
 \nu^{-1}\frac{\p^2 \Phi_{-1}}{\p z^k \p z^p}\eta^p + \nu^{-1} g_{kq}\bar \eta^q + \nu^{-1}g_{kp \bar q}\theta^p \bar \theta^q - \nu^{-1}\frac{\p^2 \Phi_{-1}}{\p z^k \p z^p}\ast \eta^p\\
- \bar\eta^q \ast (\nu^{-1} g_{kq}) + \bar \theta^q \ast \left(\nu^{-1}g_{kp \bar q}\right) \ast \theta^p + g^{qp} g_{kp \bar q} =\\
\nu^{-1}g_{kp \bar q}\theta^p \bar \theta^q + \bar \theta^q \ast \left(\nu^{-1}g_{kp \bar q}\right) \ast \theta^p - g^{qp} g_{kp \bar q} =0.\\
\end{eqnarray*}
The second equation in (\ref{E:superlong}) can be proved similarly. The last four equations in (\ref{E:superlong}) follow immediately from formulas (\ref{E:superlr}) and Proposition~ \ref{P:superpot}. 
\end{proof}

Observe that $e^{X + X'} = \mathbf{g}$. According to formula (\ref{E:strdens}), Proposition \ref{E:superdens} implies that
\[
                     \mathbf{g}\, dz d\bar z d\eta d\bar \eta d\theta d\bar \theta
\]
is a supertrace density for the product $\ast$ on $TU \oplus \Pi TU$. 

Denote by $\gamma$ the global fiberwise $(1,1)$-form on $TM$ given in local coordinates by the formula
\begin{equation}\label{E:gammaf}
                \gamma = \nu^{-1}g_{kl} d\eta^k \wedge d\bar\eta^l.
\end{equation}
The global fiberwise volume form $\gamma^m$ on $TM$ is given locally by a scalar multiple of $\nu^{-m}\mathbf{g}d\eta d\bar \eta$.
We assume that $d\beta = dzd\bar z d\theta d\bar \theta$ is the globally defined canonical Berezin density on $\Pi TM$ 
\footnote{If $\alpha = f(z,\bar z)dz^1 \wedge \ldots \wedge dz^m \wedge d \bar z^1 \wedge \ldots \wedge d\bar z^m$ is a compactly supported volume form on $U$, denote by $\hat\alpha = f(z,\bar z) \theta^1 \ldots \theta^m \bar\theta^1 \ldots \bar\theta^m$ the corresponding function on  $\Pi TU$.
Then
\[
             \int_{\Pi TU} \hat\alpha \, d\beta = \int_U \alpha.
\]}. We introduce a global supertrace density of the star product ~$\ast$ on $TM \oplus \Pi TM$ by the formula
\begin{equation}\label{E:globmu}
     \mu := \frac{1}{m!}\left(\frac{i}{2\pi} \gamma\right)^m d\beta.
\end{equation}

\begin{lemma}\label{L:prodber}
  For any formal functions $F, G$ on  $TM \oplus \Pi TM$ such that $F$ or $G$ is compactly supported over $TM$ the following identity holds,
\[
     \int F \ast G \, \mu = \int F\, \I^{-1}(G) \mu.
\]
\end{lemma}
\begin{proof}
 It suffices to prove the lemma on a coordinate chart $U \subset M$ for $F \in (C^\infty_0(TU)[\theta,\bar\theta])((\nu))$ and $G = b \ast a$, where $a=a(z,\eta,\theta)$ is holomorphic and $b=b(\bar z, \bar\eta,\bar\theta)$ is antiholomorphic on $TU \oplus \Pi TU$. Then
\begin{eqnarray*}
    \int F\, \I^{-1}(G)\, \mu = \int F\, \I^{-1}(b \ast a)\, \mu = 
  \int Fba \, \mu =  \int (-1)^{|a|(|F| + |b|)}aFb \, \mu \\
= \int (-1)^{|a|(|F| + |b|)}a \ast F \ast b \, \mu =
 \int F \ast b \ast a \, \mu = \int F \ast G \, \mu.
\end{eqnarray*}
\end{proof}

In the rest of this section we fix a contractible coordinate chart $U \subset M$ and a superpotential (\ref{E:superpot}) on $TU \oplus \Pi TU$.

\begin{lemma}\label{L:berdx}
The following identity holds,
\[
\I^{-1} \left(\frac{dX}{d\nu}\right) = \frac{dX}{d\nu} + \frac{m}{\nu}.
\]
\end{lemma}
\begin{proof}
We have, using (\ref{E:sbertran}), (\ref{E:lbullf}), (\ref{E:rbullf}), and (\ref{E:lbartheta}),
\begin{eqnarray*}
\I\left(\frac{dX}{d\nu}\right) = \I\left(\frac{d\Phi}{d\nu} - \frac{1}{\nu^2}\frac{\p \Phi_{-1}}{\p z^k}\eta^k  - \frac{1}{\nu^2}\bar\eta^l \frac{\p \Phi_{-1}}{\p \bar z^l}  + \frac{1}{\nu^2} \bar\theta^l g_{kl} \theta^k\right)=\\
\frac{d\Phi}{d\nu} - \frac{1}{\nu^2}\frac{\p \Phi_{-1}}{\p z^k}\ast \eta^k  - \frac{1}{\nu^2}\bar\eta^l \ast \frac{\p \Phi_{-1}}{\p \bar z^l}  + \frac{1}{\nu^2} \bar\theta^l g_{kl}\ast \theta^k = \frac{dX}{d\nu} - \frac{m}{\nu},
\end{eqnarray*}
whence the lemma follows.
\end{proof}
One can construct a local $\nu$-derivation analogous to (\ref{E:separdelta}) for the star product~$\ast$ on $TU \oplus \Pi TU$,
\[
       \delta = \frac{d}{d\nu} + \frac{d X}{d\nu} - R_{\frac{d X}{d\nu}}.
\]
\begin{theorem}\label{T:traceid}
  For any function $F \in (C^\infty_0(TU)[\theta,\bar\theta])((\nu))$ the following identity holds,
\[
      \frac{d}{d\nu} \int F \, \mu = \int \delta(F) \, \mu.
\]
\end{theorem}
\begin{proof}
We have by Lemmas \ref{L:prodber} and \ref{L:berdx},
\begin{eqnarray*}
  \int \delta(F) \, \mu = \int \left( \frac{dF}{d\nu} + \frac{dX}{d\nu} F - F \ast \frac{dX}{d\nu}\right) \, \mu =\\
\int \left(\frac{dF}{d\nu} + \frac{dX}{d\nu} F - F\I^{-1}\left( \frac{dX}{d\nu}\right)\right) \, \mu =\\
\int \left(\frac{dF}{d\nu} - \frac{m}{\nu}F \right)\, \mu = \frac{d}{d\nu} \int F \, \mu.
\end{eqnarray*}
\end{proof}

{\it Remark.} The statement of Theorem \ref{T:traceid} remains valid if the derivation~$\delta$ is modified by an inner derivation, i.e., replaced with the derivation $\delta + L_f - R_f$ for any $f \in C^\infty(TU \oplus \Pi TU)((\nu))$.

\section{The standard filtration}

Let $M$ be a complex manifold and $U \subset M$ be an open subset. Denote by $\P_k(U)$ the space of fiberwise homogeneous polynomial functions of degree $k$ on $TU \oplus \Pi TU$. Then the space 
$\P(U):=\prod_{k\geq 0} \P_k(U)$ of formal series $f = f_0 + f_1 + \ldots$, where $f_k \in \P_k(U)$, can be interpreted as the space of functions on the formal neighborhood of the zero section of the bundle $TU \oplus \Pi TU$.  If $U$ is a coordinate chart, $\P(U)$ is identified with $C^\infty(U)[[\eta,\bar\eta,\theta,\bar\theta]]$. We set $\Q(U) := \P(U)((\nu))$.

Given $i \in \Z$, denote by $\F^i(U)$ the space of formal series of the form 
\begin{equation}\label{E:ffromf}
f = \sum_{j=i}^\infty \sum_{r = -\infty}^{\lfloor j/2\rfloor} \nu^r f_{r, j-2r},
\end{equation}
where $f_{r,k} \in \P_k(U)$. Since $\F^{i+1}(U) \subset \F^i(U)$, $\{\F^i(U)\}$ is a descending filtration on the space
\[
\F(U) := \bigcup_{i \in \Z} \F^i(U). 
\]
This filtration is induced by a grading $\deg$ such that $\deg \nu = 2$ and $\deg f = k$ for $f \in \P_k(U)$. We denote by $\mathrm{fdeg} \, f$ the filtration degree of an element $f \in \F(U)$. These filtration and grading will be called {\it standard}. We have $\Q(U) \subset \F(U)$ and set $\Q^i(U) := \F^i(U) \cap \Q(U)$. We say that an element (\ref{E:ffromf}) of $\F(U)$ is compactly supported over $U$ if for each pair of indices $r,k$ there exists a compact $K_{r,k} \subset U$ such that $f_{r,k}$ vanishes on 
\[
T(U \setminus K_{r,k}) \oplus \Pi T(U \setminus K_{r,k}).
\]
We will write $\Q:= \Q(M), \F:= \F(M)$, etc.

If $U$ is a coordinate chart, a differential operator on $C^\infty(U)[[\eta,\bar\eta,\theta,\bar\theta]]$ has coefficients from that space and partial derivatives in the variables $z, \bar z, \eta,\bar\eta,\theta,\bar\theta$. One can define a differential operator on the space $\P$ using a partition of unity subject to a cover of $M$ by coordinate charts. A natural formal differential operator on $\Q$ is an operator $A =  A_0 + \nu A_1 + \ldots,$ where $A_r$ is a differential operator of order not greater than~$r$ on $\P$. Since in local coordinates  $\deg \p/\p \eta = \deg \p/\p \bar\eta =\deg \p/\p \theta =\deg \p/\p \bar\theta = -1$, we see that $\fdeg  A_r \geq -r$.

 \begin{lemma}\label{L:homfdo}
A natural formal differential operator $A$ on the space $\Q$ is of standard filtration degree at least zero, $A = A^0 + A^1 + \ldots$,
where $\deg A^i = i$. The homogeneous component $A^i$ is a formal differential operator of order not greater than $i$. It can be written as
\begin{equation}\label{E:ari}
                A^i = \sum_{r=0}^i \nu^r A_r^{i-2r},
\end{equation}
where $\deg A_r^j = j$.  The operator $A$ naturally extends to the space $\F$ and respects the standard filtration.
\end{lemma}
\begin{proof}
Let $A =  A_0 + \nu A_1 + \ldots$ be a natural formal differential operator on $\Q$. Since $\fdeg A_r \geq -r$, we have $\fdeg(\nu^r A_r) \geq r$, whence it follows that $\fdeg A \geq 0$. Thus we can write $A = A^0 + A^1 + \ldots$, where $\deg A^i = i$.  Each differential operator $A_r, r \geq 0,$ can be written as
\[
      A_r = \sum_{j = -r}^\infty A_r^j,
\]
where $A_r^j$ is a differential operator of order not greater than $r$ and with $\deg A_r^j = j$. Then we obtain that
\[
                  A = \sum_{r=0}^\infty \sum_{j=-r}^\infty \nu^r  A_r^j = \sum_{r=0}^\infty \sum_{i=r}^\infty \nu^r  A_r^{i-2r}  = \sum_{i=0}^\infty \sum_{r=0}^i \nu^r A_r^{i-2r},
\]
which implies (\ref{E:ari}). We see that $A^i$ is a formal differential operator of order not greater than $i$. Therefore it acts upon the space $\F$ and raises the standard filtration degree by $i$. It follows that the operator $A$ naturally extends to this space and respects the standard filtration.
\end{proof}

Let $A$ be a natural formal differential operator on $\Q$ and 
\begin{equation}\label{E:oscexp}
   K = \nu^{-1} K_{-1} + K_0 + \ldots 
\end{equation}
be an even element of $\Q$ treated as a multiplication operator with respect to the fiberwise Grassmann multiplication. Then $[K,A]$ is a natural operator on $\Q$ as well. We will consider two special cases when the series
\begin{equation}\label{E:natconj}
                    e^K A e^{-K} := \sum_{n=0}^\infty \frac{1}{n!}(\ad (K))^n A
\end{equation}
defines a natural formal differential operator on $\Q$.
\begin{lemma}\label{L:natconj}
Let $A = A_0 + \nu A_1 + \ldots$ be a natural formal differential operator on $\Q$ and $K \in \Q$ be an even element given by (\ref{E:oscexp}). Then in the following two cases the operator (\ref{E:natconj}) is natural:
\begin{enumerate}[(i)]
\item if $\fdeg K \geq 0$;
\item if $\fdeg K \geq -1$ and $\deg A_r = 0$ for all $r$.
\end{enumerate}
\end{lemma}
\begin{proof}
Since the operator $A$ is natural, we have
\[
     e^K A e^{-K} = \sum_{r=0}^\infty \sum_{n=0}^r \frac{\nu^r}{n!}(\ad (K))^n A_r.
\]
Each summand
\begin{equation}\label{E:summd}
  \frac{\nu^{r}}{n!}(\ad (K))^n A_r
\end{equation}
is a natural differential operator. Using in case $(i)$ that $\fdeg A_r \geq -r$, we see that in both cases $(i)$ and $(ii)$  the filtration degree of  (\ref{E:summd}) is at least~ $r$ and the operator $e^K A e^{-K}$ is given by a series convergent in the topology induced by the standard filtration. It is well defined on the space $\Q$, which implies the statement of the lemma.
\end{proof}

Let $\star$ be a star product with separation of variables with classifying form $\omega$ on a pseudo-K\"ahler manifold $(M, \omega_{-1})$ and $\ast$ be the star product on $TM \oplus \Pi TM$ defined by the local superpotentials (\ref{E:superpot}) as in Section~\ref{S:tmpitm}. The star product $\ast$ induces a star product on $\Q$ which will be denoted by the same symbol. It was proved in \cite{JGPsubm} that a star product with separation of variables on a split supermanifold is natural. Therefore, for $f=\nu^p f_p + \nu^{p+1}f_{p+1} +\ldots  \in \Q$ the operators $\nu^{-p} L_f$ and $\nu^{-p} R_f$ on $\Q$ are natural, extend to the space $\F$, and respect the standard filtration. It follows that the operators $L_f$ and $R_f$ extend to $\F$ as well. 
\begin{proposition}\label{P:bimod}
The superalgebra $(\Q,\ast)$ is a filtered algebra with respect to the standard filtration.
 The space $\F$ is a filtered superbimodule over the superalgebra $(\Q,\ast)$.
\end{proposition}
\begin{proof}
It suffices to prove that the superalgebra $(\Q(U),\ast)$ is a filtered algebra for $U$ a coordinate chart and $\F(U)$ is a filtered bimodule over it. An element $f \in C^\infty(U)$  identified with its lift to $TU \oplus \Pi TU$ lies in $\Q(U)$. One can see from formula (\ref{E:lbullf}) that the operator $L_f$ leaves invariant $\F(U)$ and all filtration spaces $\F^i(U)$. The operators $\L_{\eta^k} = \eta^k$ and $L_{\theta^k} = \theta^k$ leave invariant $\F(U)$ and increase the filtration degree by 1. It follows from  formula (\ref{E:lbullf}) that
\[
      \bar\theta^l = g^{lk} (g_{kq} \bar\theta^q) =  g^{lk} \ast (g_{kq} \bar\theta^q).
\]
Using formula (\ref{E:lbartheta}) we get that
\[
            L_{\bar\theta^l} = L_{g^{lk}} L_{g_{kq} \bar\theta^q} = L_{g^{lk}} \left(g_{kq} \bar\theta^q + \nu \frac{\p}{\p\theta^k}\right).
\]
Therefore, the operator $ L_{\bar\theta^l}$  leaves invariant $\F(U)$ and increases the filtration degree by 1. We have from (\ref{E:superlr}) that
\[
    L_{\frac{\p X}{\p z^k}} = \frac{\p \Phi}{\p z^k} + \frac{1}{\nu}\left(\frac{\p^2 \Phi_{-1}}{\p z^k \p z^p} \eta^p + g_{kq}\bar\eta^q + g_{kp \bar q}\theta^p \bar \theta^q\right) + \frac{\p}{\p z^k}.
\]
Using Corollary \ref{C:fastg}, we get that
\begin{eqnarray}\label{E:lgbareta}
  L_{g_{kq}\bar\eta^q} = g_{kq}\bar\eta^q + \nu \left(\frac{\p \Phi}{\p z^k} - L_{\frac{\p \Phi}{\p z^k}}\right) + \nu \frac{\p}{\p z^k} +   \frac{\p^2 \Phi_{-1}}{\p z^k \p z^p} \eta^p -\nonumber\\
 \eta^p L_{ \frac{\p^2 \Phi_{-1}}{\p z^k \p z^p}} + g_{kp \bar q}\theta^p \bar \theta^q - \theta^p L_{\Gamma_{kp}^s}\left(g_{sq}\bar\theta^q + \nu \frac{\p}{\p \theta^s}\right),
\end{eqnarray}
where $\Gamma_{kp}^s = g_{kp \bar q}g^{\bar  qs}$ is the Christoffel symbol of the Levi-Civita connection.

Given a function $f \in C^\infty(U)$, we see from  (\ref{E:lbullf}) that the operator $f - L_f$ increases the filtration degree by 1. It follows that the operator $L_{g_{kq} \bar\eta^q}$ leaves invariant $\F(U)$ and increases the filtration degree by 1. 
We have from formula (\ref{E:lbullf}) that
\[
     \bar\eta^l = g^{lk} (g_{kq} \bar\eta^q) =  g^{lk} \ast (g_{kq} \bar\eta^q) \mbox{, whence }  L_{\bar\eta^l } = L_{g^{lk}} L_{g_{kq} \bar\eta^q}.
\]
It implies that the operator $L_{\bar\eta^l }$ also leaves invariant $\F(U)$ and increases the filtration degree by 1.  The elements of $C^\infty(U)((\nu))$ and the variables $\eta,\bar\eta,\theta,\bar\theta$ generate the algebra $(\Q(U), \ast)$.  Therefore, $(\Q(U),\ast)$ is a filtered algebra and $\F(U)$ is a filtered left supermodule over $(\Q(U),\ast)$. Similar statements can be proved for the graded right $\ast$-multiplication operators which imply that $\F(U)$ is also a filtered right supermodule over $(\Q(U),\ast)$.
\end{proof}

We will use the symbol $\ast$ to denote the left and the right actions of the algebra $(\Q, \ast)$ on $\F$. 
Let $\J_r$ be the right submodule of $\F$ whose elements written in local coordinates are of the form
\[
                   u = \eta^k A_k + \theta^k  B_k  = \eta^k \ast A_k + \theta^k \ast B_k,
\]
where $A_k, B_k \in \F$, and let $\J_l$ be the left submodule of $\F$ whose elements are locally of the form
\[
                   u = C_l \bar\eta^l + D_l \bar\theta^l  = C_l \ast\bar\eta^l + D_l \ast\bar\theta^l,
\]
where $C_l, D_l \in \F$. The definitions of the submodules $\J_l$ and $\J_r$ do not depend on the choice of local holomorphic coordinates.

\section{The Lie superalgebra $\langle \chi,\tilde\chi,\sigma\rangle$ }

Let $(M, \omega_{-1})$ be a pseudo-K\"ahler manifold and $\star$ be a star product with separation of variables on $M$ with classifying form $\omega$. Recall that the function $\psi = \nu^{-1}\psi_{-1}$ on $TM \oplus \Pi TM$  was defined by (\ref{E:psimin}). Let $\ast$ be the star product with separation of variables on $TM \oplus \Pi TM$ determined by the product $\star$  and the function $\psi$, as described in Section~\ref{S:tmpitm}. In the rest of this paper we will use global functions 
\[
\varphi_{-1}, \varphi = \nu^{-1}\varphi_{-1}, \chi, \tilde\chi, \mbox{ and }\hat\omega
\]
on $TM \oplus \Pi TM$ given in local coordinates by the formulas
\begin{eqnarray}\label{E:global}
   \varphi_{-1} = \eta^k g_{kl} \bar\eta^l, \chi = \nu^{-1}\eta^k g_{kl} \bar\theta^l, \tilde\chi = \nu^{-1}\theta^k g_{kl} \bar\eta^l, \\
    \mbox{ and } \hat\omega = i\theta^k\frac{\p^2 \Phi}{\p z^k \p \bar z^l} \bar\theta^l, \nonumber
\end{eqnarray}
where $\Phi$ is a potential of $\omega$. Denote by $\sigma$ the $\ast$-supercommutator of the odd functions $\chi$ and $\tilde \chi$,
\[
          \sigma = \left[\chi,\tilde\chi \right]_\ast.
\]
The formal functions $\chi, \bar\chi$, and $\sigma$ are formal analogues of symbols of the operators $\bar \p, \bar \p^\ast$, and of the Laplace operator, respectively, on the $(0,\ast)$-forms with values in a holomorphic line bundle on $M$ used in the heat kernel proof of the index theorem for the Dirac operator $\bar\p + \bar\p^\ast$ (see \cite{BGV}).

\begin{proposition}
The formal functions $\chi$ and $\tilde\chi$ are nilpotent with respect to the star product $\ast$, $\chi \ast \chi = 0$ and $\tilde\chi\ast\tilde\chi=0$. Equivalently, they satisfy the supercommutator relations $[\chi,\chi]_\ast = 0$ and $[\tilde\chi,\tilde\chi]_\ast = 0$. The formal function $\sigma$ is given by the formula
\begin{equation}\label{E:symbsig}
       \sigma = \varphi + i\hat\omega.
\end{equation}
In particular, $\sigma \in \Q^0$.
\end{proposition}
\begin{proof}
Let $U \subset M$ be a contractible coordinate chart. We get from formulas  (\ref{E:lbartheta})  and (\ref{E:lgbareta}) that in local coordinates
\begin{eqnarray*}
  L_\chi = \chi + \eta^k \frac{\p}{\p\theta^k} \mbox{ and } L_{\tilde\chi} = \tilde\chi + \theta^k\left(\frac{\p \Phi}{\p z^k} - L_{\frac{\p \Phi}{\p z^k}}\right) +\\
   \theta^k \frac{\p}{\p z^k} +
  \frac{1}{\nu} \theta^k\eta^p\left( \frac{\p^2 \Phi_{-1}}{\p z^k \p z^p} -  L_{ \frac{\p^2 \Phi_{-1}}{\p z^k \p z^p}} \right).
\end{eqnarray*}
We see from these formulas that
\begin{equation}\label{E:chitilchi}
\chi \ast \chi = 0 \mbox{  and } \chi \ast \tilde\chi = \chi\tilde\chi + \varphi.
\end{equation}
According to formula (\ref{E:lbullf}), for any function $f \in C^\infty(U)$ we have 
\[
   (f - L_f)\chi = - (D^k f) g_{kq} \bar \theta^q = - \frac{\p f}{\p \bar z^q}\bar\theta^q  \mbox{ and } (f - L_f)\tilde\chi = 0.
\]
We obtain from these formulas that
\begin{equation}\label{E:tilchichi}
\tilde\chi \ast\chi = \tilde\chi \chi + i \hat\omega \mbox{ and }  
\tilde\chi \ast \tilde\chi = \theta^k \theta^p g_{kp \bar q} \bar\eta^q = 0.  
\end{equation}
Formula (\ref{E:symbsig}) follows from (\ref{E:chitilchi}) and (\ref{E:tilchichi}).
\end{proof}
The functions $\chi, \bar\chi$, and $\sigma$ form a basis in the Lie superalgebra~$\langle \chi,\tilde\chi,\sigma\rangle$ equipped with the supercommutator $[\cdot,\cdot]_\ast$. The element $\sigma$ generates its supercenter.  We define an element $S = \nu^{-1} S_{-1} \in \Q$ such that
\[
     S_{-1} = - \varphi_{-1} + \psi_{-1} \mbox{ and thus } S = -\varphi + \psi.
\]
The $\deg$-homogeneous component of $\sigma$ of degree zero is $-S$. One can show by a direct calculation that
\begin{eqnarray}\label{E:oplsigma}
L_\sigma = \frac{1}{\nu} \eta^k g_{kl} \bar \eta^l + \eta^k \left(\frac{\p \Phi}{\p z^k} - L_{\frac{\p \Phi}{\p z^k}}\right ) + \hskip 3cm\nonumber \\
\frac{1}{\nu}\eta^k \eta^p \left(\frac{\p^2 \Phi_{-1}}{\p z^k \p z^p} - L_{\frac{\p^2 \Phi_{-1}}{\p z^k \p z^p}}\right)  +  \eta^k \frac{\p}{\p z^k} + \frac{1}{\nu}\eta^k g_{kp\bar q} \theta^p \bar \theta^q \\
- \theta^k \left(\frac{1}{\nu} \eta^p L_{\Gamma^s_{kp}} + L_{\frac{\p^2 \Phi}{\p z^k \p \bar z^l} g^{ls}}\right)\left(g_{sq} \bar \theta^q + \nu \frac{\p}{\p \theta^s} \right).\nonumber
\end{eqnarray}
Since $\fdeg \sigma = 0$, it follows from Proposition \ref{P:bimod} that $\mathrm{fdeg} L_\sigma =0$ and $\mathrm{fdeg} R_\sigma =0$. One can see from (\ref{E:oplsigma}) that the range of the operator $L_\sigma$ lies in the submodule ~$\J_r$. Similarly,  the range of the operator $R_\sigma$ lies in the submodule ~$\J_l$.  The series
\[
     e^{\pm S} := \sum_{r \geq 0} \frac{\nu^{-r}}{r!} \left(\pm S_{-1}\right)^r
\]
are well defined elements of $\F$. We have $\deg S = \deg e^{\pm S} =0$.  Denote by $\E$ and $\bar \E$ the global holomorphic and antiholomorphic fiberwise Euler operators on $TM \oplus \Pi TM$, respectively. In local coordinates,
\[
      \E = \eta^p \frac{\p}{\p \eta^p} + \theta^p \frac{\p}{\p \theta^p} \mbox{ and } \bar\E = \bar\eta^q \frac{\p}{\p \bar\eta^q} + \bar\theta^q \frac{\p}{\p \bar\theta^q}. 
\]
Let $A^0$ and $B^0$ denote the $\deg$-homogeneous components of degree zero of the operators $L_\sigma$ and $R_\sigma$, respectively. One can see from (\ref{E:oplsigma}) and the corresponding formula for $R_\sigma$ that
\begin{equation}\label{E:azero1}
      A^0 =   - S - \E = e^{-S} (-\E) e^S \hskip 3cm \\ \mbox{ and }
  B^0 =   - S - \bar\E = e^{-S} (-\bar\E) e^S.
\end{equation}

\section{The subspace $\K$}\label{S:subk}

\begin{lemma}\label{L:luniq}
The following statements hold.
\begin{enumerate}[(i)]
\item If $F \in \J_r$ and $L_\sigma F =0$, then $F=0$. 
\item If $F \in \J_l$ and $R_\sigma F =0$, then $F=0$. 
\item If $F \in \J_l + \J_r$ and $(L_\sigma + R_\sigma) F=0$, then $F=0$.
\end{enumerate}
\end{lemma}
\begin{proof}
 Assume that $F \in \J_r$ satisfies the equation $L_\sigma F = 0$ and is nonzero. Then there exists $p \in \Z$ such that
$F = \sum_{i = p}^\infty F^i$, where $\deg F^i = i$ and $F^p$ is nonzero. We have $A^0 F^p = 0$, where $A^0$ is the zero degree component of $L_\sigma$. We see from (\ref{E:azero1}) that
\begin{equation}\label{E:esfp}
      \E \left(e^S F^p \right) = 0.
\end{equation}
Since $F^p$ is a nonzero element of $\J_r$, we get that $e^S F^p$ is also a nonzero element of $\J_r$, which contradicts (\ref{E:esfp}). Statements (ii) and (iii) can be proved similarly.
\end{proof}

The ranges of the operators $\E$ and $\bar\E$ lie in the spaces $\J_r$ and $\J_l$, respectively. 
The restriction of the operator $\E$ to $\J_r$ is invertible. Denote by $\E^{-1} : \J_r \to \J_r$ its inverse. One can define similarly the operators $\bar\E^{-1}: \J_l \to \J_l$ and $(\E + \bar\E)^{-1} : \J_l + \J_r \to \J_l + \J_r$.

\begin{lemma}\label{L:uniqsol}
 Given $G \in \ker\E\subset \F$ and $H \in \J_r$, the equation
\begin{equation}\label{E:neweq}
     \E e^S F = H
\end{equation}
has a unique solution $F \in \F$ such that $F-G\in\J_r$,
\begin{equation}\label{E:neweq1}
      F = e^{-S}(G + \E^{-1} H).
\end{equation}
\end{lemma}
\begin{proof}
Clearly, (\ref{E:neweq1}) is a solution of (\ref{E:neweq}). Assume that $F$ is a solution of (\ref{E:neweq}) satisfying $F-G\in\J_r$. Since $S \in \J_r$, we have $e^{S}F-G\in\J_r$. Therefore, $e^{S}F-G = \E^{-1} H$, which implies the uniqueness.
\end{proof}
Given $G \in \ker \bar \E \subset \F$ and $H \in \J_l$, the unique solution of the equation $ \bar\E e^S F = H$ is $F = e^{-S}(G + \bar\E^{-1} H)$. Also, given $G \in C^\infty(M)((\nu))$ and $H \in \J_l + \J_r$, one can find the unique solution of the equation $(\E + \bar \E) e^S F = H$.
\begin{proposition}\label{P:rlift}
Given an element $G \in \ker\E \subset \F$, there exists a unique element $F \in \F$ such that
\[
                 L_\sigma F = 0 \mbox{ and } F - G \in \J_r.
\]
If $G = \sum_{i \geq p} G^i$, where $\deg G^i = i$ and $G^p$ is nonzero, then $F = \sum_{i \geq p} F^i$, where $\deg F^i = i$ and $F^p = e^{-S}G^p$.
\end{proposition}
\begin{proof}
Assume that $G = \sum_{i \geq p} G^i$ and $G^p$ is nonzero. We will be looking for $F$ of the form $F = \sum_{i \geq p} F^i$ with $F^i - G^i \in \J_r$ for $i \geq p$.
Writing $L_\sigma = A^0 + A^1 + \ldots$, where  $\deg A^i = i$ and $A^0 = e^{-S}(-\E) e^S$, we rewrite the equation $L_\sigma F = 0$ as the system of equations
\begin{equation}\label{E:lzerok1}
       \E e^S F^k = e^S \sum_{i = 1}^{k-p} A^i F^{k-i}, k \geq p,  
\end{equation}
where the right-hand side of (\ref{E:lzerok1}) is zero when $k=p$. 
The condition that $F - G \in \J_r$ is equivalent to the condition that $F^k - G^k \in \J_r$ for all $k$. By Lemma \ref{L:uniqsol}, $F^p = e^{-S}G^p$. Since the range of each operator $A^i$ lies in~$\J_r$, the terms $e^S F^k$ for $k > p$ can be uniquely determined by induction from system (\ref{E:lzerok1}) using Lemma \ref{L:uniqsol}. The solution $F$ is unique by Lemma~\ref{L:luniq}.
\end{proof}
One can prove similarly the following two propositions.
\begin{proposition}
Given an element $G \in \ker\bar\E \subset \F$, there exists a unique element $F \in \F$ such that
\[
                 R_\sigma F = 0 \mbox{ and } F - G \in \J_l.
\]
If $G = \sum_{i \geq p} G^i$, where $\deg G^i = i$ and $G^p$ is nonzero, then $F = \sum_{i \geq p} F^i$, where $\deg F^i = i$ and $F^p = e^{-S}G^p$.
\end{proposition}

\begin{proposition}\label{P:lift}
Given an element $f \in C^\infty(M)((\nu)) \subset \F$, there exists a unique element $F \in \F$ such that
\[
                 (L_\sigma + R_\sigma) F = 0 \mbox{ and } F - f \in \J_l + \J_r.
\]
If $f = \sum_{r\geq p} \nu^r f_r$, where $f_r \in C^\infty(M)$ for $r \geq p$ and $f_p$ is nonzero, then $F = \sum_{i \geq 2p} F^i$ and $F^{2p} = \nu^p e^{-S}f_p$.
\end{proposition}

\begin{proposition}
   An element $F \in \F$ satisfies the equation
\begin{equation}\label{E:lplusr}
                     (L_\sigma + R_\sigma) F = 0
\end{equation}
if and only if it satisfies the equations
\begin{equation}\label{E:landr}
       L_\sigma F = 0 \mbox{ and } R_\sigma F = 0.
\end{equation}
\end{proposition}
\begin{proof}
Equations (\ref{E:landr}) imply (\ref{E:lplusr}). Assume that condition (\ref{E:lplusr}) holds and set $G := L_\sigma F$. Since the operators $L_\sigma$ and $R_\sigma$ commute, we have
\[
     (L_\sigma + R_\sigma)G = (L_\sigma + R_\sigma)L_\sigma F = L_\sigma (L_\sigma + R_\sigma)F =0.
\]
Since the range of the operator $L_\sigma$ lies in $\J_r \subset \J_l + \J_r$, we obtain from part ~ (iii) of Lemma~ \ref{L:luniq} that $L_\sigma F = 0$. One can prove similarly that $R_\sigma F = 0$.
\end{proof}

\begin{theorem}\label{T:threeequiv}
The following conditions on $F \in \F$ are equivalent:
\begin{enumerate}
\item $L_\sigma F = 0$.
\item $L_{\chi} F = 0 \mbox{ and } L_{\tilde\chi} F = 0$.
\item On any coordinate chart, $L_{\bar\eta^l} F = 0 \mbox{ and } L_{\bar\theta^l} F = 0$ for all $l$.
\end{enumerate}
\end{theorem}
\begin{proof}
Clearly, (3) $\Rightarrow$ (1) and (2). Also, (2) $\Rightarrow$ (1) because of the supercommutation relation $\left[\chi,\tilde\chi \right]_\ast = \sigma$.
Let us prove that (1)  $\Rightarrow$ (2). Set $G:= L_{\chi} F$. We have
\[
     L_\sigma G = L_\sigma L_{\chi} F = L_{\chi}L_\sigma F = 0.
\]
Locally,
\[
    G = \eta^k  L_{\nu^{-1}g_{kl} \bar\theta^l} F,
\]
hence $G \in \J_r$. Lemma \ref{L:luniq} implies that $G= L_{\chi} F=0$. The statement that $ L_{\tilde\chi} F = 0$ can be proved similarly. It remains to prove that (1), (2) $\Rightarrow$ (3). Recall that $\psi$ is the function on $TM \oplus \Pi TM$ given by (\ref{E:psimin}).  Assume that $L_\sigma F=0$ and represent $F$ as the sum 
\[
F = F_0 + F_1 + \ldots + F_m,
\]
where $F_i$ is such that $e^\psi F_i$ is homogeneous of degree $i$  in the variables $\theta$. Inspecting (\ref{E:oplsigma}), one can see that the operator $e^\psi L_\sigma e^{-\psi}$ does not contain the variables $\bar\theta$ and is of degree zero with respect to the variables $\theta$. Since $L_\sigma F=0$, it follows that
\[
      \left(e^\psi L_\sigma e^{-\psi}\right) e^\psi F_i = 0 \mbox{ for } 0 \leq i \leq m.
\]
Therefore, $L_\sigma F_i = 0$ for $0 \leq i \leq m$. Clearly, $F_i \in \J_r$ for $i \geq 1$. By Lemma \ref{L:luniq}, the equation $ L_\sigma F_i = 0$ for $i \geq 1$ has only the zero solution. Therefore, $F = F_0$, so that $\tilde F:=e^\psi F$ does not depend on the variables~ $\theta$. We have for all $p$ that
\[
     L_{\nu^{-1}g_{pq}\bar\theta^q} F = e^{-\psi}\frac{\p}{\p \theta^p} e^\psi F = 0.
\]
One gets from the formula $\bar\theta^l = \nu g^{lk} \ast \left(\nu^{-1}g_{kq}\bar\theta^q \right)$ that
\[
     L_{\bar\theta^l} F = L_{\nu g^{lk}} L_{\nu^{-1}g_{kq}\bar\theta^q}F = 0.
\]
In order to prove that $L_{\bar\eta^l}F=0$, we use that $L_{\tilde\chi}F = \theta^p  L_{\nu^{-1} g_{pq} \bar\eta^q} F = 0$. We see from (\ref{E:lgbareta}) that the operator
\[
      e^{\psi} L_{\nu^{-1} g_{pq} \bar\eta^q} e^{-\psi}
\]
does not contain the variables $\bar\theta$ and is of degree zero with respect to the variabes $\theta$. Since $\tilde F=e^\psi F$ does not depend on the variables~ $\theta$, the element
\[
      W_p:= e^{\psi} L_{\nu^{-1} g_{pq} \bar\eta^q} F = e^{\psi} L_{\nu^{-1} g_{pq} \bar\eta^q} e^{-\psi}\tilde F
\]
also does not depend on the variables $\theta$ for any $p$. We have 
\[
\theta^p W_p = e^\psi L_{\tilde\chi} F = 0, 
\]
hence $W_p=0$ for all $p$. One can show as above that it implies that $L_{\bar\eta^q} F = 0$ for all $q$.
\end{proof}
One can prove similarly the following theorem.
\begin{theorem}\label{T:threeequiv1}
The following equations on $F \in \F$ are equivalent:
\begin{enumerate}
\item $R_\sigma F = 0$.
\item $R_{\chi} F = 0 \mbox{ and } R_{\tilde\chi} F = 0$.
\item On any coordinate chart,  $R_{\eta^k} F = 0 \mbox{ and } R_{\theta^k} F = 0$  for all $k$.
\end{enumerate}
\end{theorem}

For any open subset $U \subset M$ we denote by $\K(U)$ the subspace of $\F(U)$ of elemens $F$  satisfying the following equivalent conditions:
\begin{enumerate}[(a)]
\item $(L_\sigma + R_\sigma) F = 0$;
\item $L_\sigma F = 0$ and $R_\sigma F = 0$;
\item On any coordinate chart on $U$, $\bar\eta^l \ast F = \bar\theta^l \ast F = F \ast \eta^k = F \ast \theta^k = 0$.
\end{enumerate}

 According to Proposition \ref{P:lift}, for every element $f \in C^\infty(U)((\nu))$ there exists a unique element $F \in \K(U)$ such that $F - f \in \J_l + \J_r$. It will be denoted by $K_f$. The mapping $f \mapsto K_f$ is a bijection from $C^\infty(U)((\nu))$ onto~$\K(U)$.  We set $\K := \K(M)$.

\begin{lemma}\label{L:techn}
  Assume that $f,g \in C^\infty(M)((\nu)), G \in \F,$ and $G - g \in \J_l + \J_r$. Then
\[
    K_f \ast G = K_f \ast g \mod  \J_l \mbox{ and } G \ast K_f  = g \ast K_f \mod \J_r.
\]
\end{lemma}
\begin{proof}
One can write in local coordinates
\[
G = g + \eta^k \ast A_k + B_l \ast \bar\eta^l + \theta^k \ast C_k + D_l \ast \bar\theta^l
\] 
for some $A_k, B_l, C_k, D_l \in \F$. We have
\[
    K_f \ast G = K_f \ast g + K_f \ast B_l \ast \bar\eta^l + K_f \ast D_l \ast \bar\theta^l.
\]
Therefore, $K_f \ast G = K_f \ast g \mod  \J_l$. The second statement can be proved similarly.
\end{proof}

We introduce the following notation,
\[
                  \varepsilon := K_\unit \in \K.
\]
We have $\varepsilon - \unit \in \J_l + \J_r$. One can prove a stronger statement.
\begin{proposition}\label{P:vareps}
  The element $\varepsilon$ is such that $\varepsilon - \unit \in \J_l \cap \J_r$.
\end{proposition}
\begin{proof}
It follows from Proposition \ref{P:rlift} that there exists a unique element $F \in \F$ such that $L_\sigma F = 0$ and 
\begin{equation}\label{E:fminu}
F - \unit \in \J_r. 
\end{equation}

Set $G := R_\sigma  F$. Since $\sigma \in \J_l \cap \J_r$ and $\J_r$ is a right submodule, we obtain from (\ref{E:fminu}) that $G = R_\sigma F \in \J_r$.
We also have that
\[
        L_\sigma G = L_\sigma R_\sigma F = R_\sigma L_\sigma F = 0.
\]
Lemma \ref{L:luniq} implies that $G = R_\sigma F = 0$, whence $F \in \K$. We see from (\ref{E:fminu}) that $F = \varepsilon$ and therefore $\varepsilon - \unit \in \J_r$. One can prove similarly that $\varepsilon - \unit \in \J_l$.
\end{proof}

\begin{theorem}\label{T:varkap}
 There exists an even element 
\[
      \varkappa = \nu^{-1} \varkappa_{-1} + \varkappa_0 + \ldots \in \J_l \cap \J_r
\]
such that $\varepsilon = e^{-(S +\varkappa)}$ and $\mathrm{fdeg} \, \varkappa \geq 1$.
\end{theorem}
\begin{proof}
It follows from Propositions \ref{P:lift} and \ref{P:vareps} that $\varepsilon - \unit \in \J_l \cap \J_r$,  $\mathrm{fdeg} \, \varepsilon = 0$, and the $\deg$-homogeneous component of $\varepsilon$ of degree zero is $e^{-S}$. Therefore, there exists an element $\varkappa\in\F$ such that $\varepsilon = e^{-(S+\varkappa)},\ \varkappa \in \J_l \cap \J_r$, and $\mathrm{fdeg} \, \varkappa \geq 1$.  We will write $\varkappa = \varkappa^1 + \varkappa^2 + \ldots$, where $\deg \varkappa^i = i$.  We will prove by induction on $i$ that the $\nu$-filtration degree of $\varkappa^i$ as at least~ $-1$. Since the star product $\ast$ is natural, the $\nu$-filtration degree of $\nu\sigma$ is zero, and $\mathrm{fdeg} \, (\nu\sigma)=2$, we get that the operator $\nu(L_\sigma + R_\sigma)$ is natural and $\fdeg\nu(L_\sigma + R_\sigma) = 2$.
By Lemma ~\ref{L:natconj}, the operator
\[
              C := \nu e^S (L_\sigma + R_\sigma) e^{-S}
\]
is natural. Since $\fdeg C =2$, we can write $C = C^2 + C^3 + \ldots$, where $\deg C^i = i$. By Lemma \ref{L:homfdo}, $C^i$ is a differential operator of order not greater than $i$. It follows from equation (\ref{E:azero1}) that
\[
          C^2 = -\nu(\E + \bar\E).
\]
The condition that $(L_\sigma + R_\sigma)\varepsilon =0$ implies that $(e^\varkappa C e^{-\varkappa})\unit =0$. We rewrite this equation as follows,
\begin{equation}\label{E:natopc}
      \sum_{i = 2}^\infty \sum_{j = 0}^i \frac{1}{j!}\left((\ad \varkappa)^j C^i\right) \unit = 0.
\end{equation}
Extracting the $\deg$-homogeneous component of degree 3 from (\ref{E:natopc}), we get
\[
      \nu(\E + \bar \E)\varkappa^1 + C^3\unit = 0.
\]
Since the range of the operator $C$ lies in $\J_l + \J_r$, we obtain that
\[
                  \varkappa^1 =    - \nu^{-1}(\E + \bar \E)^{-1} (C^3\unit), 
\]
whence the $\nu$-filtration degree of $\varkappa^1$ is at least -1. Now assume that for $d >1$ the $\nu$-filtration degree of $\varkappa^i$ is at least $-1$ for all $i < d$.
Extracting the $\deg$-homogeneous component of degree $d+2$ from (\ref{E:natopc}), we get that
\begin{equation}\label{E:lngvark}
    \sum_{i=2}^{d+2} \sum_{j=0}^i \frac{1}{j!} \sum_{k_1 + \ldots k_j = d+2-i} \left(\ad(\varkappa^{k_1} ) \ldots \ad(\varkappa^{k_j})C^i\right) \unit = 0.
\end{equation}
Equation  (\ref{E:lngvark}) contains $\varkappa^i$ for $i \leq d$. The only summand in (\ref{E:lngvark}) containing $\varkappa^{d}$ is
\[
             \left(\ad(\varkappa^{d})C^2\right)\unit = \nu(\E + \bar \E)\varkappa^{d}.
\]
By the induction assumption, the other summands in (\ref{E:lngvark}) are of 
$\nu$-fil\-tration degree at least zero. Since all other summands in  equation (\ref{E:lngvark}) lie in $\J_l + \J_r$, the element $\varkappa^d$ is uniquely determined by this equation and its 
 $\nu$-filtration degree is at least $-1$, which implies the statement of the theorem.
\end{proof}

\section{The algebras $\A$ and $\B$}

Let $(M, \omega_{-1})$ be a pseudo-K\"ahler manifold and $\star$ be a star product with separation of variables on $M$ with classifying form $\omega$. In this section we fix  a contractible coordinate chart $U \subset M$.  Let $\Phi_{-1}$ and $\Phi$ be potentials of $\omega_{-1}$ and $\omega$ on $U$, respectively, $g_{kl}$ be the metric tensor given by (\ref{E:metric}), and $\ast$ be the star product with separation of variables on $TU \oplus \Pi TU$ determined by the potential (\ref{E:superpot}) written as
\[
                  X = \Phi + Y,
\]
where 
\[
    Y := \frac{1}{\nu}\left(\frac{\p \Phi_{-1}}{\p z^p}\eta^p +  \frac{\p \Phi_{-1}}{\p \bar z^q}\bar\eta^q + g_{pq}\theta^p\bar\theta^q\right).           
\]
In this section we will define two subalgebras, $\A$ and $\B$, of the algebra $(\Q(U),\ast)$ and describe their action on the space $\K(U)$.

We lift differential operators on $U$ to $TU \oplus \Pi TU$ using the trivialization $TU \oplus \Pi TU \cong U \times \C^{m|m}$ induced by the choice of local coordinates on $U$. Their lifts commute with the multiplication operators by the variables $\eta,\bar\eta,\theta,\bar\theta$ and the operators $\p/\p \eta, \p/\p \bar \eta, \p/\p\theta,\p/\p\bar\theta$.
\begin{lemma}\label{L:elre}
Given $f \in C^\infty(U)((\nu))$, the operators
\[
       e^{-Y} L_f^\star e^Y \mbox{ and }  e^{-Y} R_f^\star e^Y
\]
are a left and a right $\ast$-multiplication operators on the space $\Q(U)$, respectively.
\end{lemma}
\begin{proof}
We have $Y \in \Q^{-1}(U)$ and $\fdeg L^\star_f = \fdeg R^\star_f = \fdeg f$. If the $\nu$-filtration degree of $f$ is $p$, then, according to Lemmas \ref{L:natconj} and \ref{L:homfdo}, the even formal differential operators $\nu^{-p}e^{-Y} L_f^\star e^Y$ and $\nu^{-p}e^{-Y} R_f^\star e^Y$ are natural and act on $\Q(U)$ and $\F(U)$. The operator $e^{-Y} L_f^\star e^Y$ commutes with the multiplication operators by the antiholomorphic variables $\bar z,\bar\eta,$ and $\bar\theta$ and the operators
\begin{eqnarray*}
     e^{-Y} \left(\frac{\p\Phi}{\p \bar z^l} + \frac{\p}{\p \bar z^l}\right) e^Y = \frac{\p X}{\p \bar z^l} + \frac{\p}{\p \bar z^l},  e^{-Y} \left(\frac{\p}{\p \bar \eta^l}\right) e^Y = \frac{\p X}{\p \bar \eta^l} + \frac{\p}{\p \bar \eta^l},\\
  \mbox{ and } e^{-Y} \left(\frac{\p}{\p \bar \theta^l}\right) e^Y = \frac{\p X}{\p \bar \theta^l} + \frac{\p}{\p \bar \theta^l}.
\end{eqnarray*}
Therefore, it is a left $\ast$-multiplication operator on the space $\Q(U)$. Similarly, $e^{-Y} R_f^\star e^Y$ is a right $\ast$-multiplication operator on $\Q(U)$.
\end{proof}
Given $f \in C^\infty(U)((\nu))$, we define two even elements of $\Q(U)$,
\[
      \alpha(f) := \left(e^{-Y} R_f^\star e^Y\right)\unit \mbox{ and } \beta(f) := \left(e^{-Y} L_f^\star e^Y\right)\unit.
\]
We have
\[
     R_{\alpha(f)} = e^{-Y} R_f^\star e^Y  \mbox{ and } L_{\beta(f)} = e^{-Y} L_f^\star e^Y,
\]
whence it follows that $\alpha,\beta: (C^\infty(U)((\nu)),\star) \to (\Q(U),\ast)$ are  injective homomorphisms. Their images are subalgebras of $(\Q(U),\ast)$ which will be denoted by $\A$ and $\B$, respectively.
\begin{lemma}\label{L:alphaf}
   For $f \in C^\infty(U)((\nu))$, 
\[
                  \alpha(f) - f \in \J_l + \J_r \mbox{ and } \beta(f) - f \in \J_l + \J_r.
\]
\end{lemma}
\begin{proof}
The lemma follows directly from the definitions of $\alpha(f)$ and ~$\beta(f)$.
\end{proof}
\begin{lemma}\label{L:abcomm}
Given $f \in C^\infty(U)((\nu))$, the element $\alpha(f)$ $\ast$-commutes with the variables $\bar\eta$ and $\bar\theta$ and $\beta(f)$  $\ast$-commutes with $\eta$ and $\theta$.
\end{lemma}
\begin{proof}
The operator $R_f^\star$ commutes with the variables $\bar\eta$ and $\bar\theta$. Therefore, $ R_{\alpha(f)} = e^{-Y} R_f^\star e^Y$ commutes with the operators $R_{\bar\eta^l} = \bar\eta^l$ and $R_{\bar\theta^l} = \bar\theta^l$. It follows that $\alpha(f)$ $\ast$-commutes with the variables $\bar\eta$ and $\bar\theta$. Similarly, $\beta(f)$ 
$\ast$-commutes with $\eta$ and $\theta$.
\end{proof}

\begin{corollary}\label{C:kinvar}
The left action of the algebra $\A$ and the right action of the algebra $\B$ on $\F(U)$ leave $\K(U)$ invariant. 
\end{corollary}
\begin{proof}
The corollary follows immediately from the definition of the space $\K(U)$.
\end{proof}
\begin{lemma}
   Given $f \in C^\infty(U)((\nu))$,
\begin{equation}\label{E:alphavareps}
                    \alpha(f) \ast \varepsilon = K_f \mbox{ and } \varepsilon \ast \beta(f) = K_f.
\end{equation}
\end{lemma}
\begin{proof}
  By Proposition \ref{P:vareps}, $\varepsilon - \unit \in \J_l$, whence $\alpha(f) \ast \varepsilon - \alpha(f) \in \J_l \subset \J_l + \J_r$. We get from Lemma \ref{L:alphaf} that $\alpha(f) \ast \varepsilon - f \in \J_l + \J_r$ which, according to Corollary \ref{C:kinvar}, implies the first equation in (\ref{E:alphavareps}). The proof of the second equality is similar.
\end{proof}

\begin{proposition}
  Given $f,g  \in C^\infty(U)((\nu))$,
\[
      \alpha(f) \ast K_g = K_{f \star g} \mbox{ and } K_f \ast \beta(g) = K_{f \star g}.
\]
\end{proposition}
\begin{proof}
  We have
\[
     \alpha(f) \ast K_g = \alpha(f) \ast \alpha(g) \ast \varepsilon = \alpha(f \star g) \ast\varepsilon = K_{f \star g}. 
\]
The proof of the second equality is similar.
\end{proof}
\begin{corollary}\label{C:fmod}
Given $f,g  \in C^\infty(U)((\nu))$,
\[
    f \ast K_g = f \star g \mod (\J_l+\J_r) \mbox { and } K_g \ast f = g \star f  \mod (\J_l+\J_r).
\]
\end{corollary}
\begin{proof}
Since $\alpha(f) - f \in \J_l +\J_r$, we have by Lemma \ref{L:techn} that
\[
    f \ast K_g = \alpha(f) \ast K_g = K_{f \star g} = f \star g \mod(\J_l +\J_r).
\]
The proof of the second statement is similar.
\end{proof}

Given $f \in C^\infty(U)((\nu))$, we denote by $\tilde\alpha(f)$ the element of $\Q(U)$ which does not depend on the variables $\bar\eta,\bar\theta$ and is such that $\alpha(f) - \tilde\alpha(f) \in \J_l$, that is,
\[
\tilde\alpha(f) = \alpha(f)|_{\bar\eta=\bar\theta=0}. 
\]
We get from Lemma \ref{L:alphaf} that $\tilde\alpha(f) - f \in \J_r$.
Then for $g \in C^\infty(U)((\nu))$ we have
\begin{equation}\label{E:tildea}
                      \tilde\alpha(f) \ast K_g = \alpha(f) \ast K_g = K_{f \star g} \mbox{ and } K_g \ast \tilde\alpha(f) = K_g \ast f.
\end{equation}
We define similarly an element $\tilde\beta(f) := \beta(f)|_{\eta=\theta=0}$ of $\Q(U)$ which satisfies the condition $\tilde\beta(f)-f \in \J_l$ and is such that 
\begin{equation}\label{E:tildeb}
       K_g \ast \tilde\beta(f) = K_g \ast \beta(f) = K_{g \star f} \mbox{ and } \tilde\beta(f) \ast K_g = f \ast K_g.
\end{equation}
Then we introduce an element
\begin{equation}\label{E:kappa}
              \kappa(f) := \tilde\alpha(f) + \tilde\beta(f) - f.
\end{equation}
\begin{proposition}\label{P:kcomm}
Given $f,g  \in C^\infty(U)((\nu))$,
\[
      \kappa(f) \ast K_g - K_g \ast \kappa(f) = K_{f\star g - g\star f}.
\]
\end{proposition}
\begin{proof}
  We have from (\ref{E:tildea}) and (\ref{E:tildeb}) that
\begin{eqnarray*}
    \kappa(f) \ast K_g - K_g \ast \kappa(f) = (\tilde\alpha(f) + \tilde\beta(f) - f )\ast K_g -\\
 K_g \ast (\tilde\alpha(f) + \tilde\beta(f) - f) = K_{f \star g} - K_{g \star f}.
\end{eqnarray*}
\end{proof}

\section{An evolution equation}

\begin{lemma}\label{L:etk}
 Given a nonzero complex number $k$ and a nonnegative integer $l$, the equation
\begin{equation}\label{E:poft}
           \left(e^{-kt} p(t)\right)' = e^{-kt} t^l
\end{equation}
has a unique polynomial solution
\begin{equation}\label{E:psol}
    p(t) = - \frac{1}{k} \sum_{r=0}^l \left(\frac{1}{k} \frac{d}{dt} \right)^r t^l.
\end{equation}
\end{lemma}
\begin{proof}
  Equation (\ref{E:poft}) is equivalent to the following one,
\[
       p'(t) - k p(t) = t^l.
\]
Since $k \neq 0$, it can be rewritten as follows,
\[
       \left(1 - \frac{1}{k}\frac{d}{dt}\right) p(t) = - \frac{1}{k} t^l.
\]
Using the identity
\[
    \left(1 - \frac{1}{k}\frac{d}{dt}\right)\sum_{r=0}^l \left(\frac{1}{k} \frac{d}{dt} \right)^r  = 1 -  \left(\frac{1}{k}\frac{d}{dt}\right)^{l+1},
\]
we see that for the polynomial (\ref{E:psol}),
\[
     \left(1 - \frac{1}{k}\frac{d}{dt}\right) p(t) = - \frac{1}{k}\left( 1 -  \left(\frac{1}{k}\frac{d}{dt}\right)^{l+1}\right)t^l =  - \frac{1}{k} t^l.
\]
Since $k \neq 0$, the homogeneous equation
\[
     p'(t) - k p(t)=0
\]
has no nonzero polynomial solutions. Therefore (\ref{E:psol}) is a unique polynomial solution of (\ref{E:poft}).
\end{proof}

We consider the solutions of the evolution equation
\begin{equation}\label{E:evol}
              \frac{d}{dt} F = L_\sigma F
\end{equation}
on the space $\F$ of the form $F(t) = F^{k}(t) + F^{k+1}(t) + \ldots$, where $k \in \Z$ and $\deg F^i(t) =i$ for $i \geq k$. Each $\deg$-homogeneous component $F^i(t)$ of $F$ admits an expansion in the powers of $\nu$,
\begin{equation}\label{E:compi}
  F^i(t) = \sum_{r = - \infty}^{\lfloor i/2 \rfloor} \nu^r F_r^{i-2r}(t),
\end{equation}
where $(\E + \bar\E)F_r^j(t) = j F_r^j(t)$.

It follows from formula (\ref{E:symbsig}) that the operator $L_{\nu\sigma}$ is natural and $\fdeg L_{\nu\sigma} = 2$, so that $L_{\nu\sigma} = A^2 + A^3 + \ldots$, where $\deg A^i=i$ and $A^2 =  - \nu(\E + S) = e^{-S} (-\nu\E) e^S$. Observe that
\begin{equation}\label{E:azero}
      \nu\frac{d}{dt} - A^2  = e^{-S}e^{-t\E}\left(\nu\frac{d}{dt}\right) e^{t\E}e^S.
\end{equation}

\begin{lemma}\label{L:zero}
Equation (\ref{E:evol}) has a unique solution $F(t)$ with the initial condition $F(0) = 0$, the zero solution.
\end{lemma}
\begin{proof}
Assume that $F(t) = \sum_{i \geq p} F^i(t)$ with $\deg F^i (t) = i$ and nonzero $F^p(t)$ is a nontrivial solution of (\ref{E:evol}) with the initial condition $F(0) = 0$. Then
\[
     \left(\nu\frac{d}{dt} - A^2\right)F^p = 0.
\]
We have from (\ref{E:azero}) that
\[
     \frac{d}{dt} \left( e^{t\E} e^S F^p(t) \right) = 0.
\]
Therefore, $e^{t\E} e^S F^p(t)$ does not depend on $t$. Since $F^p(0) = 0$, it follows that $e^{t\E} e^S F^p(t) = 0$ for all $t$, whence $F^p(t) =0$ for all $t$. This contradiction proves the lemma.
\end{proof}

\begin{theorem}\label{T:evo}
Equation (\ref{E:evol}) with the initial condition $F(0) = \unit$ on the space $\F$ has a unique solution, $F(t) = F^0(t) + F^1(t) + \ldots$, where $\deg F^i (t) = i$ and
\[
                F^0(t) = e^{(e^{-t}-1)S}.
\]
The component $F^i(t)$  can be expressed as (\ref{E:compi}), where for each pair $(j,r)$ the function $F_r^j(t)$ is a finite sum
\[
   F_r^j(t) = \sum_{k,l \geq 0} e^{-kt} t^l F_{r, k, l}^j
\]
such that $(\E + \bar\E)F_{r, k, l}^j = j F_{r, k, l}^j$ and $F_{r, k, l}^j = 0$ if $k =0$ and $l > 0$. In particular,
\[
     \lim_{t \to \infty} F(t) = \sum_{i=0}^\infty \sum_{r= -\infty}^{\lfloor i/2 \rfloor} \nu^r F_{r,0,0}^{i-2r}.
\]
\end{theorem}
\begin{proof} We will be looking for a solution $F(t)$ of filtration degree zero. Equation (\ref{E:evol}) can be rewritten as the system
\begin{equation}\label{E:rewrit}
       \left(\nu\frac{d}{dt}-A^2\right) F^l (t)= \sum_{i = 1}^l A^{i+2} F^{l-i}(t), \ l \geq 0,
\end{equation}
with the initial conditions $F^0(0) = \unit$ and $F^l(0)=0$ for $l > 0$. For $l=0$ the right-hand side of (\ref{E:rewrit}) is zero. We have from formulas (\ref{E:azero}) and (\ref{E:rewrit}) that
\begin{equation}\label{E:conjug}
     \nu\frac{d}{dt}\left(e^{t\E} e^S F^l(t)\right) = e^{t\E} e^S \sum_{i = 1}^l A^{i+2} F^{l-i}(t).
\end{equation} 
Thus, $e^{t\E} e^S F^0(t)$ does not depend on $t$. Since $F^0(0) = \unit$, we have
\[
    e^{t\E} e^S F^0(t) = e^S.
\]
It follows that
\begin{equation}\label{E:fzero}
      F^0(t) = e^{(e^{-t}-1)S}  =  \sum_{k= 0}^\infty \frac{1}{k! \, \nu^k} (e^{-t}-1)^k (S_{-1})^k.
\end{equation}

We will prove the theorem by induction on $i$.  We see from (\ref{E:fzero}) that the statement of the theorem holds for $i=0$. Assume that it holds for all $i < l$ for $l \geq 1$. 
Since $F^l(0)=0$ for $l \geq 1$, we obtain from (\ref{E:conjug}) that
\begin{equation}\label{E:ind}
     F^l(t) =\nu^{-1} e^{-S} e^{-t\E} \int_0^t e^{\tau\E} e^S \sum_{i = 1}^l A^{i+2} F^{l-i}(\tau) d\tau.
\end{equation} 
According to Lemma \ref{L:homfdo}, the operator $A^j$ can be written as
\[
     A^j = \sum_{r=0}^j \nu^r A_r^{j-2r},
\]
where $\deg A_r^k = k$. The component $F_r^{l-2r}(t)$ of $F^l(t)$ will be expressed as the sum
\begin{equation}\label{E:lsum}
     F_r^{l-2r}(t) = \sum_{(i,a,b,u,v)} F^{i,a,b}_{u,v}(t),
\end{equation}
where
\begin{equation}\label{E:lamb}
      F^{i, a,b}_{u,v}(t) = \frac{1}{a!} \left(-S_{-1}\right)^a e^{-t\E} \int_0^t e^{\tau\E} \frac{1}{b!} \left(S_{-1}\right)^b  A_u^{i+2-2u} F_v^{l-i-2v}(\tau) d\tau
\end{equation}
and the sum in (\ref{E:lsum}) is over the tuples $(i,a,b,u,v)$ such that 
\[
1 \leq i \leq l, a, b \geq 0, v \leq \left\lfloor\frac{l-i}{2}\right\rfloor, 0 \leq u \leq i+2,\mbox{ and } u + v - a - b= r + 1.
\]
In particular, this sum is finite. According to the induction assumption, the function $F_v^{l-i-2v}(\tau)$ in (\ref{E:lamb}) is a finite sum of expressions
\[
     e^{-k\tau} \tau^l P,
\]
where $P\in C^\infty(TM \oplus \Pi TM)$ is polynomial on fibers, $k, l \geq 0$, and the condition $k=0$ implies that $l=0$. Consider the contribution of one such expression to (\ref{E:lamb}),
\begin{equation}\label{E:pcontrib}
    \frac{1}{a!} \left(-S_{-1}\right)^a e^{-t\E} \int_0^t e^{\tau\E} \frac{1}{b!} \left(S_{-1}\right)^b \left( e^{-k\tau} \tau^l A_u^{i+2-2u}P\right) d\tau.  
\end{equation}
We see from (\ref{E:oplsigma})  that for every pair $(r,j)$, the range of the differential operator $A_r^j$ lies in $\J_r$. Therefore, $A_u^{i+2-2u} P$ can be represented as a finite sum
\[
      A_u^{i+2-2u} P = \sum_{j =1}^N Q_j
\]
of fiberwise polynomial functions $Q_j$ such that $\E P = jP$. Consider the contribution of one such function $Q_j$ to  (\ref{E:pcontrib}),
\begin{eqnarray*}
       \frac{1}{a!} \left(-S_{-1}\right)^a e^{-t\E} \int_0^t e^{\tau\E} \frac{1}{b!} \left(S_{-1}\right)^b  \left(e^{-k\tau} \tau^l  Q_j\right) d\tau = \\
      \frac{1}{a!} \left(-S_{-1}\right)^a e^{-t\E} \left(\int_0^t e^{(b - k + j)\tau}  \tau^l d\tau\right) \frac{1}{b!} \left(S_{-1}\right)^b Q_j =\\
       \left(e^{-(b+j)t}\int_0^t e^{(b - k + j)\tau}  \tau^l d\tau\right) \frac{1}{a!} \left(-S_{-1}\right)^a \frac{1}{b!} \left(S_{-1}\right)^b Q_j.
\end{eqnarray*}
Set
\[
        K(t) := e^{-(b+j)t}\int_0^t e^{(b - k + j)\tau}  \tau^l d\tau.
\]
If $b-k+j =0$, then
\[
    K(t) =  e^{-(b+j)t}\frac{t^{l+1}}{l+1}.
\]
Since $b \geq 0$ and $j \geq 1$, we see that $b+j \geq 1$. If $b-k+j \neq 0$, then, according to Lemma \ref{L:etk}, there exists a polynomial $p(\tau)$ such that
\[
      \left(e^{(b-k+j)\tau}p(\tau)\right)' = e^{(b - k + j)\tau}  \tau^l.
\]
It follows that
\[
        K(t) =  e^{-kt} p(t) - e^{-(b+j)t} p(0).
\]
If $k=0$, then $l=0$ and
\[
     K(t) = \frac{1 - e^{-(b+j)t}}{b+j}.
\]
We have thus shown that all summands contributing to $F_r^{l-2r}(t)$ satisfy the conditions of Theorem \ref{T:evo}. Thus, $F(t)$ exists and satisfies the conditions of the theorem. By Lemma \ref{L:zero}, it is a unique solution of (\ref{E:evol}) with the initial condition $F(0)=\unit$.
\end{proof}

\begin{proposition}\label{P:commute}
If $F(t)$ is the solution of equation (\ref{E:evol}) with the initial condition $F(0)=\unit$ and $W$ is an element of the Lie superalgebra $\langle \chi,\tilde\chi,\sigma\rangle$, then
\[
     L_W F(t) = R_W F(t).
\]
The function $F(t)$ is a unique solution of the equation
\begin{equation}\label{E:revol}
     \frac{d}{dt} F(t) = R_\sigma F(t)
\end{equation}
with the initial condition $F(0) = \unit$.
\end{proposition}
\begin{proof} We use the fact that $\sigma$ lies in the supercenter of $\langle \chi,\tilde\chi,\sigma\rangle$. Set
\[
      G(t) := \left(L_W  - R_W \right) F(t).
\]
We have
\begin{eqnarray*}
     \frac{d}{dt} G(t) = \left(L_W  - R_W \right) \frac{d}{dt} F(t) = \left(L_W  - R_W \right) L_\sigma F(t) = \\
L_\sigma \left(L_W  - R_W \right)F(t) = L_\sigma G(t)
\end{eqnarray*}
and
\[
     G(0) = \left(L_W  - R_W \right) F(0) = \left(L_W  - R_W \right)\unit = 0.
\]
Lemma \ref{L:zero} implies that $G$ is the zero function, i.e., $L_W F(t) = R_W F(t)$. Therefore, $F(t)$ is a solution of equation (\ref{E:revol}) with the initial condition $F(0)=\unit$. The uniqueness of this solution can be proved as the uniqueness in Theorem \ref{T:evo}.
\end{proof}
\begin{lemma}\label{L:frestr}
The solution $F(t)$ of equation (\ref{E:evol}) with the initial condition $F(0)=\unit$ satisfies the property that $F(t) - \unit \in \J_l \cap \J_r$ for every value of $t$.
\end{lemma}
\begin{proof}
   Since the range of the operator $L_\sigma$ lies in $\J_r$, equation  (\ref{E:evol}) implies that
\[
      \frac{d}{dt}F(t) \in \J_r.
\]
Since $F(0) = \unit$, it follows that $F(t)- \unit \in \J_r$. Since the range of the operator $R_\sigma$ lies in $\J_l$, we see that equation  (\ref{E:revol}) similarly implies that $F(t) - \unit \in \J_l$.
\end{proof}
\begin{theorem}\label{T:lim}
If $F(t)$ is the solution of the evolution equation (\ref{E:evol}) with the initial condition $F(0)=\unit$, then
\begin{equation}\label{E:limeps}
     \lim_{t \to \infty} F(t) = \varepsilon.
\end{equation}
\end{theorem}
\begin{proof}
The limit in (\ref{E:limeps}) exists by Theorem \ref{T:evo}. Denote it temporarily by $Z$. It follows from Lemma \ref{L:frestr} that $Z- \unit \in \J_l \cap \J_r$. To prove the theorem it remains to show that $L_\sigma Z = 0$. We write
\[
   F(t) = \sum_{i=0}^\infty\sum_{r=-\infty}^{\lfloor i/2 \rfloor} \nu^r F_r^{i-2r}(t),
\]
where for each pair $(j,r)$ the $\deg$-homogeneous element $F_r^j(t)$ of degree~$j$ is 
expressed as a finite sum
\[
     F_r^j(t) = \sum_{k,l\geq 0} e^{-kt} t^l F^j_{r,k,l}
\]
such that $F^j_{r,k,l} \in C^\infty(TM \oplus \Pi TM)$ is polynomial on fibers, $\deg F^j_{r,k,l}=j$, and $F^j_{r,k,l}=0$ for $k=0$ and $l\geq 1$. In particular,
\[
   \frac{d}{dt}  F_r^j(t) = \sum_{k >0,l\geq 0} e^{-kt}(l t^{l-1} - k t^l) F^j_{r,k,l}.
\]
Therefore, for any pair $(j,r)$ we have that
\[
     \lim_{t \to \infty} \frac{d}{dt}  F_r^j(t) = 0.
\]
The operator $L_{\nu\sigma}$ is a natural operator of filtration degree 2. According to Lemma \ref{L:homfdo},
\[
       L_{\nu\sigma} = \sum_{i=2}^\infty \sum_{r=0}^i \nu^r A_r^{i-2r},
\]
where $\deg A_r^j = j$. Equation (\ref{E:evol}) is equivalent to the system
\begin{equation}\label{E:longsys}
       \frac{d}{dt} F_r^{i-2r} = \sum_{(j,k,p,q)} A_p^{j-2p} F_q^{k-2q}, \ i \geq 0, r \leq \lfloor i/2 \rfloor,
\end{equation}
where the summation is over the tuples $(j,k,p,q)$ satisfying the conditions $j+k = i+2, p+q = r+1, j \geq 2, 0 \leq p \leq j, k \geq 0, q \leq \lfloor k/2 \rfloor$. In particular, the sum in (\ref{E:longsys}) is finite. Taking the limit as $t \to \infty$ of both sides of (\ref{E:longsys}) we obtain a system equivalent to the equation $L_\sigma Z = 0$. It follows from Proposition \ref{P:rlift} that $Z = \varepsilon$.
\end{proof}

Let $F(t) = F^0(t) + F^1(t) + \ldots$  be the solution of equation~(\ref{E:evol}) with the initial condition $F(0)=\unit$. According to Theorem \ref{T:evo}, 
\[
F^0(t) = \exp\{(e^{-t}-1)S\}. 
\]
There exists a function $G(t) = G^0(t) + G^1(t) + \ldots$, where $\deg G^i(t)=i$, such that $\exp G(t) = F(t), \ G(0)=0$, and
\[
     G^0(t) = (e^{-t}-1)S = \nu^{-1}(e^{-t}-1)S_{-1}.
\]
Lemma \ref{L:frestr} implies that $G(t) \in \J_l \cap \J_r$.
\begin{theorem}\label{T:oscill}
The $\nu$-filtration degree of the function $G(t)$ is $-1$.
\end{theorem}
\begin{proof}
We will prove by induction on $i$ that for every $i \geq 0$ the  $\nu$-filtration degree of the function $G^i = G^i(t)$ is at least $-1$. This is true for $i=0$. Assume that this is true for $i < p$. Rewrite equation (\ref{E:evol}) in terms of $G$,
\begin{equation}\label{E:evolg}
     \frac{d}{dt} G = \left(e^{-G} L_\sigma e^G\right)\unit.
\end{equation}
Since $\fdeg \sigma = 0$, one can write $L_\sigma = A^0 + A^1 + \ldots$, where $\deg A^i = i$. According to (\ref{E:azero1}), 
\[
              A^0 = e^{-S} (-\E) e^S.
\]
Extract the component of (\ref{E:evolg}) of degree $p$:
\begin{equation}\label{E:logf}
   \frac{d}{dt} G^p = \left(\sum_{k=0}^\infty  \sum_{i_1 + \ldots +i_k + l = p} \frac{(-1)^k}{k!} (\ad G^{i_1}) \ldots (\ad G^{i_k}) A^l \right)\unit.
\end{equation}
The summands on the right-hand side of (\ref{E:logf}) containing $G^p$ have all but one $i_j = 0$ and $l=0$. They add up to
\begin{eqnarray*}
  \left(\sum_{k=1}^\infty  k \frac{(-1)^k}{k!} (\ad G^0)^{k-1} (\ad G^p) A^0\right)\unit = \left( e^{-\ad G^0} (-\ad G^p) A^0\right)\unit =\\
\left( (F^0)^{-1} \left[A^0, G^p\right] F^0\right)\unit = - \E G^p.
\end{eqnarray*} 
Equation (\ref{E:logf}) can be written as
\begin{equation}\label{E:hpright}
       \frac{d}{dt} G^p + \E G^p = H^p,
\end{equation}
where $H^p$ is the sum of all terms on the right-hand side of (\ref{E:logf}) which do not contain $G^p$. One can rewrite (\ref{E:hpright}) as follows,
\begin{equation}\label{E:exphpright}
       e^{-t\E} \left(\frac{d}{dt}\right) e^{t\E}G^p = H^p.
\end{equation}
Since the range of the operator $L_\sigma$ lies in $\J_r$, we have that $H^p \in \J_r$ and therefore
\[
              G^p (t) = e^{-t\E} \int_0^t e^{\tau \E} H^p(\tau) \, d\tau.
\]
Using the induction assumption and the fact that the operator $\nu L_\sigma$ is natural,  it is easy to check that the $\nu$-filtration degree of $H^p$ is at least~$-1$. Therefore, the $\nu$-filtration degree of $G^p$ is also at least~$-1$, which concludes the induction proof.
\end{proof}

\section{Oscillatory symbols}

Given an open set $U \subset M$, let $\hat\P(U)$ denote the subspace of $\F(U)$ of elements of the form
\[
     f = \sum_{i = p}^\infty \sum_{r = r_p}^{\lfloor i/2 \rfloor} \nu^r f_{r,i},
\]
where $p, r_p \in \Z$ and $f_{r,i} \in \P_{i-2r}(U)$. We set $\hat\P := \hat\P(M)$. One can check using Lemma \ref{L:homfdo} that a natural formal differential operator on~$\Q$ extended to $\F$ leaves invariant $\hat\P$. We define an oscillatory symbol~$F$ as an element of~$\F$ which admits a representation
\begin{equation}\label{E:oscillrep}
                  F = e^{-h} G,
\end{equation}
where $G \in \hat \P$ and $h$ is a global function on $TM \oplus \Pi TM$ such that in local coordinates $h = \nu^{-1}h_{kl}(z, \bar z) \eta^k \bar \eta^l$ and  $\left(h_{kl}(z, \bar z)\right)$ is an $m \times m$-matrix nondegenerate at every point $(z,\bar z)$. We denote by~$\O$ the space of oscillatory symbols in~$\F$. It is a union of linear spaces $\O_h$ of oscillatory symbols with a fixed function $h$. For any open subset $U \subset M$ one can similarly define the space $\O(U)$ of oscillatory symbols in $\F(U)$.

\begin{lemma}
A nonzero oscillatory symbol $F \in \O$ has a unique representation (\ref{E:oscillrep}).
\end{lemma}
\begin{proof}
Suppose that a nonzero element $F \in \O$ has two representations of the form (\ref{E:oscillrep}),
\[
     F =  e^{-h} G = e^{-\tilde h} \tilde G.
\]
Since $\deg h = \deg \tilde h=0$, we can assume that $F$ is $\deg$-homoge\-ne\-ous. Then
$\deg G = \deg \tilde G = \deg F$ and $G, \tilde G$ are formal functions on $TM \oplus \Pi TM$ polynomial on fibers whose $\nu$-degree is bounded below and above. We obtain the equality
\[
    e^{-(h - \tilde h)} G = \tilde G
\]
which holds if and only if $h = \tilde h$ and $G = \tilde G$.
\end{proof}

\begin{lemma}
A natural formal differential operator $A$ on the space $\Q$ extended to $\F$ leaves invariant each space $\O_h$ of oscillatory symbols, i.e., given an element $e^{- h} G \in \O$, there exists an element $\tilde G \in \hat\P$ such that
\[
       A \left(e^{-h} G\right) = e^{- h} \tilde G \in \O.
\]
\end{lemma}
\begin{proof}
By Lemma \ref{L:natconj}, the operator $e^{h}  A e^{- h}$ is natural and therefore leaves invariant the space $\hat\P$. We have $\tilde G = \left(e^{h}  A e^{- h}\right) G$.
\end{proof}
Let $h, w$ be global functions on $TM \oplus \Pi TM$ given in local coordinates by the formulas $h = \nu^{-1} h_{kl}(z, \bar z)\eta^k \bar \eta^l$ and  $w = \nu^{-1}w_{kl}(z, \bar z)\theta^k \bar \theta^l$, where $(h_{kl})$ is a nondegenerate matrix. Assume that $H \in \J_l \cap \J_r$ is an even element whose $\nu$-filtration degree as at least $-1$, $\fdeg H = 0$, and  the $\deg$-homogeneous component of $H$ of degree zero is $h + w$. 
\begin{lemma}\label{L:etominh}
We have $e^{-H} \in \O$. 
\end{lemma}
\begin{proof}
One can write
\[
                   H = h + w + \tilde H,
\]
where $\tilde H \in \Q^1$ and the $\nu$-flitration degree of $\tilde H$ is at least $-1$. Set 
\[
     G := \exp\{- (w + \tilde H)\}.
\]
Since $\exp\{-\tilde H\} \in \hat\P$ and $w$ is nilpotent, we see that $G \in \hat \P$. Now,
\[
                             \exp\{-H\}  =  \exp\{- h\}G,
\]
whence the lemma follows.
\end{proof}

\begin{proposition}\label{P:kfftino}
For any $f \in C^\infty(M)((\nu))$ we have $K_f \in \O$. For any $t \geq 0$ we have $F(t) \in \O$.
\end{proposition}
\begin{proof}
The proposition follows from  Lemma \ref{L:etominh}  and Theorems \ref{T:varkap} and \ref{T:oscill}.
\end{proof}

Assume that, as above,  a global function $h$ on $TM \oplus \Pi TM$ is given in local coordinates by the formula $h = \nu^{-1}h_{kl}(z, \bar z)\eta^k \bar\eta^l$, where $(h_{kl})$ is a nondegenerate matrix at every point $(z, \bar z)$. Then there exists a global differential operator $\Delta_h$ on $TM \oplus \Pi TM$ given in local coordinates by the formula
\[
      \Delta_h = \nu h^{lk} \frac{\p^2}{\p \eta^k \p \bar \eta^l},
\]
where $ \left(h^{lk}\right)$ is the matrix inverse to~$\left(h_{kl}\right)$. There exists a fiberwise endomorphism $\Lambda_h$ of the holomorphic cotangent bundle of $M$ given in local coordinates by the formula 
\begin{equation}\label{E:lambdah}
\Lambda_h = (g_{kl} h^{lp}), 
\end{equation}
where $g_{kl}$ is the pseudo-K\"ahler metric tensor.  Let
\[
      \zeta: C^\infty(TM \oplus \Pi TM) \to C^\infty(\Pi TM)
\]
denote the restriction mapping to the zero section of the vector bundle $TM \oplus \Pi TM \to \Pi TM$ (we identify the zero section with $\Pi TM$). In local coordinates, $\zeta(F) = F|_{\eta=\bar\eta=0}$. Define a global mapping 
\[
T_h: \hat \P \to C^\infty(\Pi TM)((\nu)) 
\]
by the formula
\begin{equation}\label{E:mappingt}
     T_h(G) := \det \left(\Lambda_h\right) \zeta\left(e^{\Delta_h} G \right).
\end{equation}
Since $\deg \left(\Delta_h\right) =0$, the operator $\exp\{\Delta_h\}$ acts upon each homogeneous component of $G$ as a differential operator of finite order. Let $G^j$ be the $\deg$-homogeneous component of $G$ of degree $j$. It follows that $\deg T_h(G^j) = j$. Since in local coordinates $T_h(G^j)$ does not depend on the variables $\eta, \bar\eta$,  the $\nu$-filtration degree of $T_h(G^j)$ is bounded below by $j/2 - m$, which implies that the $\nu$-degree of $T_h(G)$ is bounded below. Therefore, the mapping $T_h$ is well-defined. One can interpret (\ref{E:mappingt}) as a fiberwise formal oscillatory integral on the vector bundle $TM \oplus \Pi TM \to \Pi TM$,
\begin{equation}\label{E:formoscillint}
     \int e^{- h} G\,  \frac{1}{m!} \left( \frac{i}{2 \pi} \gamma\right)^m := 
\det \left(\Lambda_h\right) \zeta\left(e^{\Delta_h} G \right),
\end{equation}
where $\gamma$ is given by (\ref{E:gammaf}). If $h_{kl}$  is a Hermitian metric tensor, $\nu$ is a positive number, and $G \in C^\infty(TM \oplus \Pi TM)$ is fiberwise polynomial, then the integral in (\ref{E:formoscillint}) converges to the right-hand side. 

Let $U \subset M$ be a coordinate chart. Recall that $\mathbf{g} = \det (g_{kl})$ and $\log \mathbf{g}$ is any branch of the logarithm of~$\mathbf{g}$ on $U$. Define similarly  $\mathbf{h} := \det (h_{kl})$ and $\log \mathbf{h}$. 

\begin{theorem}\label{T:thid}
Let $f \in C^\infty(\Pi TU)((\nu))$ and $G \in \hat\P(U)$. The following identities hold true:
\begin{enumerate}
\item $T_h\left(f G \right) = f T_h\left(G \right) $;
\item $T_h\left( \left(\frac{\p}{\p \eta^p}- \nu^{-1}h_{pl}\bar\eta^l\right)G \right) = 0$;
\item $T_h\left( \left(\frac{\p}{\p \bar\eta^q}- \nu^{-1}h_{kq}\eta^k\right)G \right) = 0$;
\item $\frac{d}{d\nu}T_h\left( G \right) = T_h\left(\left(\frac{d}{d\nu} + \frac{h}{\nu} - \frac{m}{\nu}\right)G \right)$.
\item $\frac{\p}{\p z^p}T_h\left(G\right) = T_h\left(\left(\frac{\p}{\p z^p} - \frac{\p h}{\p z^p} + \frac{\p}{\p z^p} \log \mathbf{g}\right)G\right)$;
\item $\frac{\p}{\p\bar z^q}T_h\left(G\right) = T_h\left(\left(\frac{\p}{\p \bar z^q} - \frac{\p h}{\p \bar z^q} + \frac{\p}{\p \bar z^q} \log \mathbf{g}\right)G\right)$;
\item $\frac{\p}{\p\theta^p}T_h\left( G \right) = T_h\left(\frac{\p}{\p \theta^p} G \right)$;
\item $\frac{\p}{\p\bar\theta^q}T_h\left(G \right) = T_h\left(\frac{\p}{\p \bar\theta^q} G \right)$.
\item $\frac{\p}{\p h_{kl}}T_h\left(G \right) = T_h\left(- \frac{1}{\nu}\eta^k \bar \eta^l G \right)$.
\end{enumerate}
\end{theorem}
\begin{proof}
Statements (1), (7), and (8) are trivial. We prove (2) as follows:
\begin{eqnarray*}
    T_h\left( \left(\frac{\p}{\p \eta^p}- \frac{1}{\nu}h_{pl}\bar\eta^l\right)G \right) =  
T_h\left( e^{- \Delta_h}\left(- \frac{1}{\nu}h_{pl}\bar\eta^l \right)e^{\Delta_h} G \right) =\\
\det\left(\Lambda_h\right) \left(- \nu^{-1}h_{pl}\bar\eta^l \right) e^{\Delta_h} G \Big |_{\eta=\bar\eta = 0} = 0.
\end{eqnarray*}
Statement (3) can be proved similarly. To prove (4), we observe that
\[
     \frac{d}{d\nu}T_h\left(G \right) =  T_h\left(\left(\frac{d}{d\nu} + \nu^{-1}\Delta_h \right)G \right).
\]
It remains to show that $T_h\left(\Delta_h G \right) =  T_h\left(\left(h - m\right)G \right)$.
This identity readily follows from items (1), (2), and (3). We prove (5) as follows:
\begin{eqnarray*}
   \frac{\p}{\p z^p}T_h\left( G \right) =  \frac{\p}{\p z^p} \left(\det \left(\Lambda_h\right)e^{\Delta_h} G \Big |_{\eta=\bar\eta = 0} \right) =\hskip 3.5cm\\
 \det \left(\Lambda_h\right)\left(\frac{\p}{\p z^p} \log \det \left(\Lambda_h\right)\right) e^{\Delta_h} G \Big |_{\eta=\bar\eta = 0} +\hskip 3.5cm\\
  \det \left(\Lambda_h\right)e^{\Delta_h}\left( \nu\frac{\p h^{lk}}{\p z^p}\frac{\p^2 G}{\p \eta^k \p \bar \eta^l}\right) \Big |_{\eta=\bar\eta = 0} + \det \left(\Lambda_h\right)e^{\Delta_h} \frac{\p G}{\p z^p} \Big |_{\eta=\bar\eta = 0}\\
= T_h\left(\left(\frac{\p}{\p z^p}+ \frac{\p}{\p z^p} \log \mathbf{g}  \right)G\right) + \hskip 2cm\\
 T_h\left(\left(\nu\frac{\p h^{lk}}{\p z^p}\frac{\p^2}{\p \eta^k \p \bar \eta^l}  - \frac{\p}{\p z^p} \log \mathbf{h}   \right)G\right).
\end{eqnarray*}
Now it remains to show that
\[
     T_h\left(\left(\nu\frac{\p h^{lk}}{\p z^p}\frac{\p^2}{\p \eta^k \p \bar \eta^l}  - \frac{\p}{\p z^p} \log\mathbf{h} \right)G\right) =  T_h\left(\left( -  \frac{\p h}{\p z^p}\right)G\right).
\]
This equality can be derived from items (1), (2), (3), and the formula
\[
        \frac{\p}{\p z^p} \log \mathbf{h} = h^{lk} \frac{\p h_{kl}}{\p z^p}.
\] 
Identity~(6) can be proved similarly. In order to show (9), we use the formula
\[
             \frac{\p h^{qp}}{\p h_{kl}} = - h^{qk} h^{lp}
\]
to prove that
\[
             \frac{\p}{\p h_{kl}}  \log \det (\Lambda_h) = - h^{lk}
\]
and then use formulas (2) and (3).
\end{proof}
This theorem justifies the interpretation of the mapping $T_h$ as a formal oscillatory integral. Identities (2) and (3) can be obtained by integrating  the formal integral in (\ref{E:formoscillint}) by parts and identities (4) - (9) can be obtained by differentiating it with respect to a parameter.

Let $U \subset M$ be a coordinate chart, $(h_{kl})$ be a nondegenerate $m \times m$-matrix with elements from $C^\infty(U)$, and $(\alpha_{kl})$ be a matrix with even nilpotent elements from $C^\infty(U)[\theta,\bar\theta]$. Suppose that $f $ is a smooth function in $m^2$ complex variables such that the composition 
\[
f(h_{kl}) := f(h_{11}, \ldots, h_{mm})
\]
is defined.  Then one can define the composition of $f$ with the functions $h_{kl} + \alpha_{kl}$ using the Taylor series of $f$ which terminates due to the nilpotency of $\alpha_{kl}$,
\[
        f(h_{kl} + \alpha_{kl}) :=  e^{\alpha_{pq} \frac{\p}{\p h_{pq}}}f(h_{kl}).
\]
We set $h:= \nu^{-1}h_{kl}\eta^k \bar \eta^l$ and $\alpha:= \nu^{-1}\alpha_{kl}\eta^k \bar \eta^l$. Given $e^{- h}G \in \O(U)$, we can rewrite it as 
\begin{equation}\label{E:rewr}
e^{- (h+\alpha)}(e^{\alpha}G).
\end{equation}
 Also, we can define the matrix $\Lambda_{h + \alpha}$, the operator $\Delta_{h + \alpha}$, and the mapping $T_{h+\alpha}$ by the same formulas, because the matrix $(h_{kl} + \alpha_{kl})$ is invertible. The formal oscillatory integral (\ref{E:formoscillint}) should not change if we rewrite the integrand as (\ref{E:rewr}). This is indeed the case.
\begin{lemma}\label{L:rewr}
Given $G \in \hat\P(U)$, the following identity holds:
\[
      T_{h + \alpha}(G) = T_h\left(e^{- \alpha}G\right).
\]
\end{lemma}
\begin{proof}
Using Theorem \ref{T:thid} (9), we get that
\[
   T_{h + \alpha}(G) = e^{\alpha_{pq} \frac{\p}{\p h_{pq}}} T_h(G) = T_h\left(e^{- \alpha}G\right).
\]
\end{proof}
This lemma holds true in the global setting when $h$ and $\alpha$ are defined on $TM \oplus \Pi TM$ and $G \in \hat\P$.

Suppose that an element $e^{- h} G \in \O$ is compactly supported over $M$.
We define a formal integral of $e^{- h} G$ with respect to the density $\mu$ given in (\ref{E:globmu}) as follows,
\[
     \int e^{- h} G \, \mu := \int_{\Pi TM} T_h(G) \, d\beta.
\]
We want to show that this formal integral is a supertrace functional on the space of formal oscillatory symbols $\O$ with respect to the action of the algebra $(\Q,\ast)$.

Let $A$ be a differential operator on the space $TM \oplus \Pi TM$. There exists a differential operator $A^t$, the transpose of $A$, on that space such that for any $f,g \in C^\infty(TM \oplus \Pi TM)$ with $f$ or $g$ compactly supported over $TM$ the following identity holds,
\[
       \int_{TM \oplus \Pi TM} (Af)\cdot  g \, \mu = \int_{TM \oplus \Pi TM}(-1)^{|f| |A|} f \cdot (A^t g)\mu.
\]
The mapping $A \mapsto A^t$ is involutive and has the property that
\begin{equation}\label{E:abt}
                  (AB)^t = (-1)^{|A||B|} B^t A^t.
\end{equation}
If $A$ is a multiplication operator by a function with respect to the fiberwise Grassmann product, then~$A ~= ~A^t$. In local coordinates we have
\begin{eqnarray*}
    \left(\frac{\p}{\p z^k}\right)^t = - \frac{\p}{\p z^k} - \frac{\p}{\p z^k} \log \mathbf{g}; \hskip 1cm\\
 \left(\frac{\p}{\p\bar z^l}\right)^t = - \frac{\p}{\p\bar z^l}  - \frac{\p}{\p\bar z^l}\log \mathbf{g}; \hskip 1cm\\
\left(\frac{\p}{\p\eta^k}\right)^t = - \frac{\p}{\p\eta^k};  \left(\frac{\p}{\p\bar\eta^l}\right)^t = - \frac{\p}{\p\bar\eta^l}; \\
\left(\frac{\p}{\p\theta^k}\right)^t = - \frac{\p}{\p\theta^k}; \left(\frac{\p}{\p\bar\theta^l}\right)^t = -\frac{\p}{\p\bar\theta^l}.
\end{eqnarray*}
The mapping $A \mapsto A^t$ induces a transposition mapping on the differential operators on the space $\P$ . The transpose operator of a differential operator $A$ on $\P$ will be denoted also by $A^t$.
\begin{proposition}\label{P:transp}
 Given $f \in \Q$ and $e^{- h} G \in \O$ such that~$f$ or~$G$ is compactly supported over $M$, then for any differential operator $A$ on~$\Q$ we have
\begin{equation}\label{E:transp}
      \int (Af)\cdot  e^{- h} G \, \mu = \int (-1)^{|f| |A|} f \cdot A^t\left( e^{- h} G\right)\mu.
\end{equation}
\end{proposition}
\begin{proof}
We will prove the proposition on a coordinate chart $U \subset M$. To prove (\ref{E:transp}) for $A = \p/\p \eta^p$ we verify the identity
\[
      \int \frac{\p f}{\p \eta^p}\cdot  e^{- h} G \, \mu = -\int  f \frac{\p}{\p \eta^p} \left( e^{- h} G\right)\mu.
\]
We have
\begin{eqnarray*}
   \int \frac{\p f}{\p \eta^p}\cdot  e^{- h} G \, \mu + \int  f \frac{\p}{\p \eta^p} \left( e^{- h} G\right)\mu =\\
\int e^{- h} \left(\left(\frac{\p}{\p \eta^p}- \nu^{-1}h_{pl}\bar\eta^l\right)(fG)\right)\mu =\\
  \int_{\Pi TU} T_h\left(\left(\frac{\p}{\p \eta^p}- \nu^{-1}h_{pl}\bar\eta^l\right)(fG)\right) d\beta = 0
\end{eqnarray*}
by item (2) of Theorem \ref{T:thid}. Identity  (\ref{E:transp}) for $A = \p/\p \bar\eta^q$ follows from~ item~(3)  of Theorem \ref{T:thid}. To prove (\ref{E:transp}) for $A = \p/\p \theta^p$ we verify the identity
\[
     \int \frac{\p f}{\p\theta^p}\cdot  e^{- h} G \, \mu = -\int (-1)^{|f|} f \frac{\p}{\p\theta^p}\left( e^{- h} G\right)\mu.
\]
Using item (7) of Theorem \ref{T:thid}, we have that
\begin{eqnarray*}
  \int \frac{\p f}{\p\theta^p}\cdot  e^{- h} G \, \mu  + \int (-1)^{|f|} f \frac{\p}{\p\theta^p}\left( e^{- h} G\right)\mu=\hskip 1.5cm \\
  \int e^{- h} \frac{\p}{\p\theta^p}(fG) \, \mu = \int_{\Pi TU} T_h \left( \frac{\p}{\p\theta^p}(fG)  \right) d\beta =\\
    \int_{\Pi TU} \frac{\p}{\p\theta^p} T_h (fG)\, d\beta = 0.
\end{eqnarray*}
One can similarly verify (\ref{E:transp}) for $A = \p/\p \bar\theta^q, \p/\p z^p,$ and $\p/\p\bar z^q$ using  items (8), (5), and (6) of Theorem~\ref{T:thid}, respectively. Now the statement of the Proposition follows from (\ref{E:abt}).
\end{proof}

\begin{theorem}\label{T:traceono}
  Given $f \in \Q$ and  $e^{- h} G \in \O$, where $G$ is compactly supported over $M$, the following identity holds,
\[
     \int (L_f - R_f)\left(e^{- h} G\right) \, \mu = 0.
\]
\end{theorem}
\begin{proof}
 The condition that $\mu$ is a trace density for the star product $\ast$ on $TM \oplus \Pi TM$ is equivalent to the condition that 
\begin{equation}\label{E:tracetrans}
    \left(L_f - R_f\right)^t \unit = 0
\end{equation}
for any $f \in C^\infty(TM \oplus \Pi TM)((\nu))$. Therefore, (\ref{E:tracetrans}) holds for any $f \in \Q$. By Proposition \ref{P:transp},
\begin{eqnarray*}
     \int (L_f - R_f)\left(e^{- h} G\right) \, \mu =
 \int \unit \cdot (L_f - R_f)\left(e^{- h} G\right) \, \mu =\\
   \int \left((L_f - R_f)^t \unit \right)\left(e^{- h} G\right) \, \mu =0.
\end{eqnarray*}
\end{proof}

We introduce a functional $\tau$ on the compactly supported formal functions on $M$ by the formula
\[
     \tau(f) := \int K_f \,\mu.
\]
It follows from Proposition \ref{P:kfftino}  that it is well defined. Using partition of unity, Proposition \ref{P:kcomm}, and Theorem \ref{T:traceono} one can show that $\tau$ is a trace functional on the algebra $(C^\infty(M)((\nu)),\star)$.

\begin{theorem}\label{T:fint}
Assume that the manifold $M$ is compact and $F(t)$ is the solution of (\ref{E:evol}) with the initial condition $F(0)=\unit$. Then the identity
\begin{equation}\label{E:fint}
     \tau(\unit) = \int F(t) \, \mu
\end{equation}
holds for all $t \geq 0$.
\end{theorem}
\begin{proof}
According to Proposition \ref{P:kfftino}, the formal integral in (\ref{E:fint}) is well defined. Using evolution equation (\ref{E:evol}), Proposition~\ref{P:commute}, and Theorem~\ref{T:traceono} we get
\begin{eqnarray*}
  \frac{d}{dt} \int F(t) \, \mu = \int L_\sigma F(t) \, \mu = 
  \int (L_{\chi}L_{\tilde \chi} + L_{\tilde \chi} L_{\chi}) F(t) \, \mu =\\
  \int (L_{\chi}L_{\tilde \chi} + L_{\tilde \chi} R_{\chi}) F(t) \, \mu =
  \int (L_{\chi} - R_{\chi}) L_{\tilde \chi} F(t) \, \mu = 0.  
\end{eqnarray*}
Therefore, the integral in (\ref{E:fint}) does not depend on $t$. Now the statement of the theorem follows from Theorem \ref{T:lim}.
\end{proof}

\section{Identification of the trace functional $\tau$}

 In this section we will prove that for any contractible open subset $U \subset M$ there exists a constant $c \in \C$ such that
\[
             \tau(f) = c \int_U f \, \mu_\star
\]
for all $f \in C_0^\infty(U)((\nu))$. The local $\nu$-derivation
\[
       \delta = \frac{d}{d\nu} + \frac{d X}{d\nu} - R_{\frac{d X}{d\nu}}
\]
for the star product $\ast$ on $TU \oplus \Pi TU$ induces a derivation on the algebra $(\Q(U),\ast)$ which we also denote by $\delta$. Moreover, the action of $\delta$ on $\Q(U)$ extends to $\F(U)$ so that the Leibniz rule holds. Namely, for $f \in \Q(U)$ and $g \in \F(U)$,
\[
     \delta(f \ast g) = \delta(f) \ast g + f \ast \delta(g) \mbox{ and }  \delta(g \ast f) = \delta(g) \ast f + g \ast \delta(f).
\]
We will modify the derivation $\delta$ by an inner derivation so that the resulting derivation will leave invariant the subspace $\K(U) \subset \F(U)$. Set
\[
       w = -\tilde\beta\left(\frac{d\Phi}{d\nu}\right) + \eta^k \tilde\beta\left(\frac{1}{\nu^2}\frac{\p\Phi_{-1}}{\p z^k}\right)
\]
and define
\[
           \tilde\delta:=\delta + L_w - R_w.
\]
\begin{theorem}\label{T:tildedelta}
Given $f \in C^\infty(U)((\nu))$, we have that
\[
      \tilde\delta (K_f) = K_{\tilde\delta_\star f},
\] 
where
\[
     \tilde\delta_\star =  \frac{d}{d\nu} +  \frac{d\Phi}{d\nu} - L^\star_{ \frac{d\Phi}{d\nu}}
\]
is a derivation of $(C^\infty(U)((\nu)),\star)$.
\end{theorem}
\begin{proof}
First we will prove that the space $\K(U)$ is invariant under the action of $\tilde\delta$. We have
\begin{eqnarray*}
    \tilde\delta(\bar\eta^l) = \frac{d X}{d \nu} \ast \bar\eta^l - \bar\eta^l \ast \frac{d X}{d \nu} + w \ast \bar\eta^l - \bar\eta^l \ast w = \\
- \bar\eta^l \ast \left(\frac{d\Phi}{d\nu} - \frac{1}{\nu^2} \eta^k \frac{\p \Phi_{-1}}{\p z^k} + w\right) \mod \J_l.
\end{eqnarray*}
Given $f \in C^\infty(U)((\nu))$, we get using (\ref{E:tildeb}) that
\begin{eqnarray*}
    \left(\frac{d\Phi}{d\nu} - \frac{1}{\nu^2} \eta^k \frac{\p \Phi_{-1}}{\p z^k} + w\right) \ast K_f = \hskip 5cm\\
  \left(\frac{d\Phi}{d\nu} - \frac{1}{\nu^2} \eta^k \frac{\p \Phi_{-1}}{\p z^k} -\tilde\beta\left(\frac{d\Phi}{d\nu}\right) + \eta^k \tilde\beta\left(\frac{1}{\nu^2}\frac{\p\Phi_{-1}}{\p z^k}\right) \right) \ast K_f = 0.
\end{eqnarray*}
Thus, $\tilde\delta(\bar\eta^l) \ast K_f =0$. One can  prove similarly that $\tilde\delta(\bar\theta^l) \ast K_f = 0$. Next,
\begin{eqnarray*}
    \tilde\delta(\eta^k) = \eta^k \ast \frac{d X}{d \nu}  - \eta^k \ast \frac{d X}{d \nu} + \hskip 1.5cm\\
 w \ast \eta^k - \eta^k \ast w = w \ast \eta^k \mod \J_r.
\end{eqnarray*}
Given $f \in C^\infty(U)((\nu))$, we have using (\ref{E:tildeb}) that
\begin{eqnarray*}
  K_f \ast \tilde\delta(\eta^k)  = K_f \ast w \ast \eta^k = \hskip 2.5cm\\
 - K_f \ast \tilde\beta\left(\frac{d\Phi}{d\nu}\right) \ast \eta^k = -K_{f \star \frac{d\Phi}{d\nu}} \ast \eta^k  =0.
\end{eqnarray*}
Similarly, $K_f \ast \tilde\delta(\theta^k) = 0$. It follows that
\[
    \bar\eta^l \ast \tilde\delta (K_f) =  \bar\eta^l \ast \tilde\delta (K_f) + \tilde\delta(\bar\eta^l) \ast  K_f = \tilde\delta(\bar\eta^l \ast  K_f) =0.
\]
We prove along the same lines that
\[
     \bar\theta^l \ast \tilde\delta (K_f) = \tilde\delta (K_f) \ast \eta^k = \tilde\delta (K_f) \ast \theta^k =0.
\]
Therefore, $\tilde\delta (K_f) \in \K(U)$. Now we need to find an element $g \in C^\infty(U)((\nu))$ such that $\tilde\delta (K_f) - g \in \J_l + \J_r$. For this element we will have $\tilde\delta (K_f) = K_g$. We write explicitly
\[
\tilde\delta (K_f) = \frac{d}{d\nu}K_f + \frac{dX}{d\nu}K_f - K_f \ast \frac{dX}{d\nu} + w \ast K_f - K_f \ast w.
\]
Since $K_f = f \mod (\J_l + \J_r)$ and $\frac{dX}{d\nu} = \frac{d\Phi}{d\nu} \mod (\J_l + \J_r)$, it follows that
\begin{equation}\label{E:mod1}
     \frac{d}{d\nu}K_f + \frac{dX}{d\nu}K_f = \frac{df}{d\nu} + \frac{d\Phi}{d\nu}f \mod  (\J_l + \J_r).
\end{equation}
By Lemma \ref{L:techn} and Corollary \ref{C:fmod},
\begin{equation}\label{E:mod2}
    K_f \ast \frac{dX}{d\nu} = K_f \ast \frac{d\Phi}{d\nu} = f \star \frac{d\Phi}{d\nu} \mod  (\J_l + \J_r).
\end{equation}
Using  (\ref{E:tildeb}) we get that
\begin{eqnarray*}
    w \ast K_f - K_f \ast w = \left(-\tilde\beta\left(\frac{d\Phi}{d\nu}\right) + \eta^k \tilde\beta\left(\frac{1}{\nu^2}\frac{\p\Phi_{-1}}{\p z^k}\right)\right) \ast K_f -\\
 K_f \ast \left(-\tilde\beta\left(\frac{d\Phi}{d\nu}\right) + \eta^k \tilde\beta\left(\frac{1}{\nu^2}\frac{\p\Phi_{-1}}{\p z^k}\right)\right) =\hskip 2cm\\
- \frac{d\Phi}{d\nu} \ast K_f + \frac{1}{\nu^2} \eta^k \ast \frac{\p\Phi_{-1}}{\p z^k} \ast K_f + K_{f \star \frac{d\Phi}{d\nu}}.
\end{eqnarray*}
Therefore, by Corollary \ref{C:fmod},
\begin{equation}\label{E:mod3}
    w \ast K_f - K_f \ast w =  - \frac{d\Phi}{d\nu} \star f + f \star \frac{d\Phi}{d\nu} \mod  (\J_l + \J_r).
\end{equation}
Combining (\ref{E:mod1}), (\ref{E:mod2}), and (\ref{E:mod3}), we get
\[
    \tilde\delta (K_f) = \frac{df}{d\nu} + \frac{d\Phi}{d\nu}f - \frac{d\Phi}{d\nu} \star f = \tilde\delta_\star (f) \mod  (\J_l + \J_r),
\]
which concludes the proof.
\end{proof}

\begin{theorem}\label{T:osctrace}
  Given $e^{- h} G \in \O(U)$ such that $G$ is compactly supported over $U$, the following identity holds,
\[
      \frac{d}{d\nu} \int e^{- h} G \, \mu = \int \tilde\delta\left(e^{-h} G\right) \, \mu.
\]
\end{theorem}
\begin{proof}
Writing $\tilde\delta = \frac{d}{d\nu} + A$, we see that Theorem \ref{T:traceid} is equivalent to the fact that
\[
     (A^t \unit) \mu = \frac{d\mu}{d\nu}
\]
 which implies Theorem \ref{T:osctrace} according to Proposition \ref{P:transp}.
\end{proof}
\begin{corollary}\label{C:constc}
There exists a constant $c \in \C$ such that
\[
     \tau(f) = c  \int_U f \, \mu_\star
\]
for any $f \in C_0^\infty(U)((\nu))$.
\end{corollary}
\begin{proof}
By Theorem \ref{T:tildedelta}, we have
\begin{eqnarray*}
   \frac{d}{d\nu}\tau(f) =  \frac{d}{d\nu}\int K_f \, \mu = \int \tilde\delta (K_f) \, \mu =\\
  \int K_{\tilde\delta_\star(f)} \, \mu = \tau\left(\tilde\delta_\star(f)\right),
\end{eqnarray*}
whence the corollary follows.
\end{proof}

Corollary~\ref{C:constc} implies the following theorem.

\begin{theorem}\label{T:constmult}
If the manifold $M$ is connected, then there exists a constant $c$ such that
\[
          \tau(f) = \int K_f \, \mu = c \int_M  f \, \mu_\star
\]
for all compactly supported formal functions $f$  on $M$.
\end{theorem}

\section{Getzler's rescaling}

In this section we use the rescaling of $TM \oplus \Pi TM$ introduced by Getzler in \cite{G}. Consider the operator $\lambda_s$ on $C^\infty(TM \oplus \Pi TM)$ given in local coordinates by the formula
\[
    \lambda_s: F(z, \bar z, \eta,\bar\eta,\theta,\bar\theta) \mapsto F(z, \bar z, s\eta,s\bar\eta,s\theta,s\bar\theta).
\]
Given a function $F$ on $TM \oplus \Pi TM$ compactly supported over $TM$, it is easy to verify that
\[
                    \int \lambda_s F \, \mu
\]
does not depend on $s$. An analogous statement holds also for the formal integral when $F$ is an oscillatory symbol.
\begin{proposition}\label{P:indep}
Given $e^{- h}G \in \O$ such that $G \in \hat P$ is compactly supported over $M$, the formal integral
\[
          \int \lambda_s (e^{- h}G) \, \mu          
\]
does not depend on $s$.
\end{proposition}
\begin{proof}
We have
\[
    \int \lambda_s (e^{- h}G) \, \mu = \int_{\Pi TM} T_{s^2 h}(\lambda_s G) d\beta.
\]
Since only the component of $G$ of bidegree $(m,m)$ with respect to the odd variables $(\theta,\bar\theta)$ contributes to the integral, it suffices to prove that
\[
      T_{s^2 h}(\lambda_s G) = s^{-2m} T_h(G).
\]
Clearly, $\det (\Lambda_{s^2h}) = s^{-2m} \det (\Lambda_h)$ and $\Delta_{s^2 h} = s^{-2}\Delta_h$. Denote by $G_{k,l}$ the component of $G$ of bidegree $(k,l)$ with respect to the variables~$(\eta,\bar\eta)$. Then
\[
     \zeta \left(e^{s^{-2}\Delta_h} \lambda_s(G) \right)= \sum_{r=0}^\infty \frac{1}{r!}(\Delta_h)^r G_{r,r}
\]
does not depend on $s$, whence the proposition follows.
\end{proof}
We introduce an operator
\[
          L_\sigma^s :=s^2 \lambda_s^{-1} L_\sigma\lambda_s
\]
and a function
\[
                 G(s,t) := \lambda_s^{-1}F(s^2 t),
\]
where $F(t)$ is the solution of (\ref{E:evol}) with the initial condition $F(0)=\unit$.
\begin{lemma}
The function $G(s,t)$ is the unique solution of the equation
\begin{equation}\label{E:tevol}
                     \frac{d}{dt}G =  L_\sigma^s G
\end{equation}
with the initial condition $G|_{t=0} = \unit$ on the space $\F$. 
\end{lemma}
\begin{proof}
The lemma follows from the calculation
\begin{eqnarray*}
   \lambda_s\left( \frac{d}{dt} G\right) = \frac{d}{dt}(\lambda_s G) = \frac{d}{dt}( F(s^2t) )=\\
 s^2 \frac{dF}{dt}(s^2t) = s^2 L_\sigma F(s^2t) = s^2 L_\sigma \lambda_s G.
\end{eqnarray*}
\end{proof}
Consider the grading on the functions on $TM \oplus \Pi TM$ polynomial on fibers given by the operator $\E+ \bar \E$. For the variables $\eta,\bar\eta,\theta,\bar\theta$,
\[
     |\eta|=|\bar\eta|=|\theta|=|\bar\theta|=1
\]
(we assume that $|\nu|=0$). This grading induces an {\it ascending} filtration on the space of formal differential operators on $\Q$. The subspace of filtration degree~$d$ consists of the operators of the form $A = A_d + A_{d-1} + \ldots$, where~$A_j$ is homogeneous of degree $j$ with respect to the this grading. We call this grading the $\lambda$-grading, because 
\[
\lambda_s (A_j)\lambda_s^{-1} = s^j A_j.
\]
It follows from (\ref{E:lbullf}) that for $f \in C^\infty(M)((\nu))$ the operator $f - L_f$ has the $\lambda$-filtration degree $-1$. We see from (\ref{E:oplsigma}) that the $\lambda$-filtration degree of the operator $L_\sigma$ is 2.
Denote by $L_\sigma^0$ the homogeneous component of~$L_\sigma$ of $\lambda$-degree 2. We have
\[
     L_\sigma^0 = \lim_{s \to 0} L_\sigma^s.
\] 
The curvature $R = R_k^u$ of the K\"ahler connection on $M$ is given by the formula
$R_k^u = R^u_{kp\bar q} dz^p \wedge d \bar z^q$, where
\[
R^u_{kp\bar q} = (g_{k p \bar b} g^{\bar ba} g_{a \bar l \bar q} - g_{kp\bar l \bar q})g^{\bar lu}.
\]
In local coordinates the operator $L_\sigma^0$ is expressed as follows,
\begin{equation}\label{E:lsigmazero}
                      L_\sigma^0 = \sigma + \hat R_k^u \frac{\p}{\p \eta^u},
\end{equation}
where $\hat R_k^u := R^u_{kp\bar q}\theta^p \bar \theta^q$.

The operator $L_\sigma^s$ can be written as a series $L_\sigma^s = B_0 + s B_1 + \ldots$, where the $\lambda$-degree of $B_i$ is $2-i$ and 
$B_0 = L_\sigma^0$.
\begin{theorem}\label{T:regular}
The function $G(s,t)$ is regular at $s=0$.
\end{theorem}
\begin{proof}
We will prove by induction on $i$ that the $\deg$-homogeneous component $G^i$ of degree $i$ of the function $G$ is regular at $s=0$. We have $G^i = \lambda_s^{-1} F^i(s^2 t)$, where $F^i$ is the $\deg$-homogeneous component of degree $i$ of the function $F(t)$. We get from Theorem \ref{T:evo} that
\[
     G^0(s,t) = \exp\left\{\frac{e^{-s^2 t} - 1}{s^2}S\right\}.
\]
This function is regular at $s=0$ and
\[
     G^0(0,t) = e^{-tS}.
\]
Assume that $G^i$ is regular at $s=0$ for $i < l$. In Theorem \ref{T:evo} we used the notation $L_{\nu\sigma} = A^2 + A^3 + \ldots$, where $\deg A^i = i$. Using that
\[
    e^{-S} e^{-t\E} e^S = \exp\left\{e^{-S}(-t\E)e^{S}\right\} = e^{-t(\E + \E S)},
\]
we can rewrite formula (\ref{E:ind})  as follows:
\begin{equation}\label{E:induct}
    F^l(t) =\frac{1}{\nu} e^{-t(\E + \E S)} \int_0^t e^{\tau(\E + \E S)} \sum_{i = 1}^l A^{i+2} F^{l-i}(\tau) d\tau.
\end{equation}
Applying $\lambda_s^{-1}$ to both sides of (\ref{E:induct}), we get
\[
    \lambda_s^{-1} F^l(t) =\frac{1}{\nu} e^{-t(\E + s^{-2}\E S)} \int_0^t e^{\tau(\E + s^{-2}\E S)} \sum_{i = 1}^l ( \lambda_s^{-1}A^{i+2} \lambda_s) \lambda_s^{-1} F^{l-i}(\tau) d\tau.
\]
Replacing $t$ with $s^2 t$ and using the substitution $\tau = s^2 u$, we obtain that
\[
    G^l(s,t) = \frac{1}{\nu} e^{-t(s^2\E + \E S)} \int_0^t e^{\tau(s^2\E +\E S)} \sum_{i = 1}^l ( s^2\lambda_s^{-1}A^{i+2} \lambda_s)G^{l-i}(s,u) du.
\]
Since the $\lambda$-filtration degree of the operator $L_{\nu\sigma}$ is 2, the operator $s^2\lambda_s^{-1}A^{i+2} \lambda_s$ is regular at $s=0$. Therefore, by the induction assumption, $G^l$ is also regular at $s=0$, whence the theorem follows.
\end{proof}

The matrix
\[
      H (t) := \frac{e^{t \hat R} - 1}{\hat R}
\]
is well-defined and the elements of the matrix $H(t) - t \cdot 1$ are even and nilpotent.
Theorem \ref{T:regular} implies that the function $G(0,t)$ is a solution of the equation
\[
      \frac{d}{dt}G(0,t) = L_\sigma^0 G(0,t)
\]
with the initial condition $G(0,0)=\unit$. It is easy to prove that this solution is unique. Denote by $\tilde h$ the global function on $TM \oplus \Pi TM$ given in local coordinates by the formula
$\tilde h =  -\nu^{-1}H_k^u(t) g_{u\bar l}\eta^k \bar \eta^l$. A direct check shows that
\[
       G(0,t) = \exp\left\{it \hat\omega  - \tilde h  \right\},
\]
where $\hat\omega$ is given by (\ref{E:global}). It follows that
\[
       G(0,t) = e^{t\varphi} \tilde G(t),
\]
where $\tilde G(t) \in \hat\P$, so that $G(0,t) \in \O$ for $t \neq  0$. By (\ref{E:lambdah}),
\[
     \Lambda_{\tilde h} = - (H(t))^{-1} =  \frac{\hat R}{1 - e^{t \hat R}}.
\]
Using Lemma \ref{L:rewr}, we can calculate the following formal oscillatory integral,
\begin{eqnarray*}
    \int G(0,t)\,  \frac{1}{m!} \left( \frac{i}{2 \pi} \gamma\right)^m =
e^{ it \hat\omega} \det \left(\frac{\hat R}{1 - e^{t \hat R}}\right).
\end{eqnarray*}

Now assume that the manifold $M$ is compact and connected. According to Theorem \ref{T:fint} and Proposition \ref{P:indep},
\[
     \tau(\unit) = \int G(s,t) \, \mu
\]
for any $s \neq 0$ and $t > 0$. Passing to the limit as $s \to 0$, we obtain that
\[
     \tau(\unit) = \int G(0,t) \, \mu = \int_{\Pi TM} e^{it \hat\omega} \det \left(\frac{\hat R}{1 - e^{t \hat R}}\right) \, d\beta
\]
for any $t \neq 0$. We will set $t = -1$. Theorem \ref{T:constmult} implies that there exists a constant $c \in \C$ such that
\[
     c \int_M \mu_\star = \int_{\Pi TM} e^{-i\hat\omega} \det \left(\frac{\hat R}{1 - e^{-\hat R}}\right) \, d\beta,
\]
where $[\omega]$ is the de Rham class of $\omega$. Since the leading term of the canonical trace density $\mu_\star$ is given by (\ref{E:lead}), we see that $c=1$. 

Using the fact that the Todd genus ${Td}(M)$ of~$M$ has a de Rham representative
\[
     \det \frac{R}{1 - e^{-R}},
\]
where $R$ is the curvature of the K\"ahler connection on $M$, we obtain the following algebraic Riemann-Roch-Hirzebruch theorem for deformation quantization with separation of variables.
\begin{theorem}
Let $\star$ be a star product with separation of variables on a compact connected pseudo-K\"ahler manifold $M$ with classifying form~$\omega$. Then
\begin{equation}\label{E:hrr}
    \int_M \mu_\star = \int_M e^{-i [\omega]} \, {Td} (M),
\end{equation}
where $\mu_\star$ is the canonical trace density of $\star$.
\end{theorem}

The curvature of the K\"ahler connection on $T_{\C}M = T^{(1,0)}M \oplus T^{(0,1)}M \cong T^{(1,0)}M \oplus T^{*(1,0)}M$ is given by the matrix
\[
\begin{bmatrix}
R & 0 \\
0 & -R^t
\end{bmatrix},
\]
where we identify $T^{(0,1)}M$ and $T^{*(1,0)}M$ via the pseudo-K\" ahler metric on $M$. Therefore, the $\hat A$-genus of $M$ has a de Rham representative
\[
    \det \frac{R/2}{\sinh(R/2)} = e^{-\frac{1}{2} \tr(R)} \det \frac{R}{1 - e^{-R}}.
\]
Now the index formula (\ref{E:index}) follows from formulas (\ref{E:hrr}) and (\ref{E:class1}) and the fact that $\rho = i \tr R$.

\end{document}